\documentclass[11pt,reqno]{amsart}
\usepackage{amsmath,amsthm,amsfonts,amstext,verbatim,epsfig,psfrag,mathrsfs}
\usepackage[initials,alphabetic]{amsrefs}
\RequirePackage[OT1]{fontenc}
\RequirePackage[colorlinks,linkcolor=blue,citecolor=blue,urlcolor=blue]{hyperref}
\usepackage{amsfonts,amstext,amssymb,verbatim,epsfig,psfrag,stmaryrd,extarrows}
\usepackage{amssymb}
\usepackage{mathtools}
\usepackage{anysize}
\usepackage{hyperref}
\usepackage{extarrows}
\usepackage{stmaryrd}
\usepackage{color}
\usepackage{enumitem}
\usepackage{mathtools} %for sub/superstack
\usepackage{dsfont} % for \mathds
\usepackage{soul}
\usepackage{amsmath,amsthm,amssymb, bbm, array, graphicx, tikz, tikz-cd}
\usepackage{pgfplots}
\usepackage{bm}
\usepackage{extarrows} 
\usetikzlibrary{arrows}
\usepackage[caption = false]{subfig}
\usepackage{subcaption}
\usepackage{graphicx} % subcaption for subfigure environment
\usepackage{caption}

\numberwithin{equation}{section}
\theoremstyle{plain} 
\newtheorem{theorem}{Theorem}[section]
\newtheorem{lemma}[theorem]{Lemma}

\newtheorem{proposition}[theorem]{Proposition}

\newtheorem{remark}[theorem]{Remark}

\renewcommand{\Re}{\mathrm{Re}\,}
\renewcommand{\Im}{\mathrm{Im}\,}

\newcommand{\E}{{\mathbf E }}

\newcommand{\R}{{\mathbb R }}
\newcommand{\N}{{\mathbb N}}

\renewcommand{\P}{{\mathbf P}}
\newcommand{\C}{{\mathbb C}}

\newcommand{\LL}{{\mathcal L}}

\newcommand{\ov}{\overline}
\DeclareMathOperator{\Tr}{Tr}
\DeclareMathOperator{\erfc}{erfc}

\DeclarePairedDelimiter{\abs}{\lvert}{\rvert}

\newcommand{\ii}{\mathrm{i}}

\newcommand{\dd}{\mathrm{d}}
\newcommand{\ie}{\emph{i.e., }}

\newcommand{\cf}{\emph{c.f., }}

\newcommand{\wh}{\widehat}

\newcommand{\nc}{\normalcolor}

\newcommand{\bs}{\boldsymbol}

\def\Tr{\mathrm{Tr}}

\def\one{\mathds{1}}

\def\<{\langle}
\def\>{\rangle}

\renewcommand{\mathbf}[1]{\bs{#1}}

\newcommand{\si}{\sigma}

\newcommand{\8}{\infty}

\marginsize{25mm}{25mm}{25mm}{26mm}

\allowdisplaybreaks

\begin{document}

	\begin{minipage}{0.85\textwidth}
		\vspace{2.5cm}
	\end{minipage}
	\begin{center}
		\large\bf Large deviations for the extremal eigenvalues of Ginibre ensembles		
	\end{center}

	\renewcommand*{\thefootnote}{\fnsymbol{footnote}}	
	\vspace{0.5cm}
	
	\begin{center}
		\begin{minipage}{1.0\textwidth}
			\begin{minipage}{0.5\textwidth}
				\begin{center}
					Yuanyuan Xu\\
					\footnotesize 
					{SKLMS,~AMSS,~Chinese Academy of Sciences}\\
					{\it yyxu2023@amss.ac.cn}
				\end{center}
			\end{minipage}
			\begin{minipage}{0.5\textwidth}
				\begin{center}
					Qiang Zeng\\
					\footnotesize 
					{SKLMS,~AMSS,~Chinese Academy of Sciences}\\
					{\it qzeng@amss.ac.cn}
				\end{center}
			\end{minipage}
		\end{minipage}
	\end{center}
	
	\bigskip

	\renewcommand*{\thefootnote}{\arabic{footnote}}
	
	\vspace{5mm}

	\begin{center}
		\begin{minipage}{0.91\textwidth}\small{
				{\bf Abstract.}}
			We establish large deviation principles for the extremal eigenvalues of the Ginibre ensembles with good rate functions. In contrast to the typical estimates for the extremal eigenvalues, the large deviations for the real Ginibre ensemble come from the eigenvalues lying on the real line. Moreover, we also derive deviation estimates for the second leading term in the asymptotic expansion of the extremal eigenvalues. These polynomially small deviation estimates are universal for any i.i.d.~matrices under a mild moment condition.

		\end{minipage}
	\end{center}

	\vspace{5mm}
	
	{\footnotesize
		{\noindent\textit{Keywords}: Large deviation, extremal statistics, Ginibre ensemble}\\
		{\noindent\textit{MSC number}:  60B20, 60G55, 60G70}\\
		{\noindent\textit{Date}:  \today \\
		}
		
		\vspace{2mm}

		\thispagestyle{headings}

		\bigskip

		\normalsize

		\section{Introduction and main results}

		We consider the Ginibre matrix ensemble, \ie $n\times n$ matrices $X=\big(x_{ij}\big)_{i,j=1}^n$ with independent, identically distributed (i.i.d.) Gaussian entries,  
		%We use the normalization $\E x_{ij}=0$, 
		%$\E \abs{x_{ij}}^2= \frac{1}{n}$, \ie
		where $\sqrt{n}x_{ij}$ is a standard real or complex Gaussian random variable. We use the parameter $\beta=1,2$ to distinguish the real Ginibre ensemble~($\beta=1$) and the complex Ginibre ensemble~($\beta=2$).

		 Let $\{\sigma_i\}_{i=1}^n$  denote the eigenvalues of the Ginibre matrix $X$. The corresponding empirical spectral distribution (ESD), $\mu_n:=\frac{1}{n}\sum_{i=1}^n \delta_{\sigma_i}$, converges to the uniform distribution on the unit disk in the complex plane~\cite{Ginibre,Edel}, which is also known as the celebrated \emph{circular law}. Furthermore, the large deviation principle (LDP) from the circular law was established in~\cite{AZ98} for the real Ginibre ensemble and \cite{HP1997} for the complex Ginibre ensemble, with speed $n^2$ and a good rate function as follows. 
		\begin{theorem}[\cite{AZ98,HP1997}]\label{thm:ldp_esd}
			The empirical spectral measures $\mu_n$ of the Ginibre ensemble ($\beta=1,2$) satisfy the LDP with speed $n^2$ and good rate function given by 
			%\cor [Check $\beta$ is missing!] \nc
			\begin{align}
				I(\mu)=&\frac{\beta}{2} \Big( \int_\C |z|^2 \mu(\dd^2 z) -\int_\C \int_\C \log |z-w|  \mu(\dd^2 z)\mu(\dd^2 w) \Big)-\frac{3\beta}{8}.
				%\nonumber\\
				%	&\qquad -\frac{1}{2} \Big( \int_{|z|\leq 1} |z|^2 \dd^2 z -\int_{|z|\leq 1} \int_{|w|\leq 1} \log |z-w|  \dd^2 z \dd^2 w \Big)
			\end{align}
            Moreover, the uniform distribution $\mu_D$ on the unit disk is the unique minimizer of $I(\cdot)$ so that $I(\mu_D)=0$.
		\end{theorem}
	    
    Although  the circular law indeed holds for a wide class of   matrices with i.i.d.~entries
     %consisting of general i.i.d. random variables
     ~\cite{Girko1984, Bai97, TV10}, the LDP from the circular law is still widely open beyond the Ginibre ensembles.  
        %with very few results for certain special matrices~\cite{HP1997,AZ98}???? 
      On the other hand, the LDPs for the ESDs and extremal eigenvalues of Hermitian random matrices have been extensively studied since the seminal works \cite{BG97, BDG01}; 
      %and have remained an active subject of study recently      ; 
      see the survey paper \cite{Gui23} for more details.

\smallskip		
		
	%Despite the above development, 
	However, deviation estimates for the extremal eigenvalues of non-Hermitian matrices seem rare (with the exception of a few results discussed below).	In this paper, we focus on the study of the large deviations of
 the rightmost eigenvalue $\max_{i=1}^n \{\Re\sigma_i\}$ and the spectral radius $\max_{i=1}^n \{|\sigma_i|\}$.  Besides its mathematical interest,  the rightmost eigenvalue of a large non-Hermitian matrix is closely related to the stability properties of complex biological systems~\cite{May72}. Motivated by this, a series of works \cite{Gem86,BY86,BCCT18,BCG22} has proved that both the rightmost eigenvalue and the spectral radius converge to one as the dimension goes to infinity, at a nearly optimal speed $O(n^{-1/2+\epsilon})$~\cite{AEK19}. For the Ginibre ensembles, both of them
 %the spectral radius of the Ginibre ensemble %converges to one~\cite{AEK19} with
		have the following three-term asymptotic expansions~\cite{R03,RS14,Bender,AkemannP,maxRe_Gin}
		\begin{align}
           \max_{i=1}^n \Re\sigma_i \stackrel{\mathrm{d}}{=}& 1+\sqrt{\frac{\gamma'_n }{4n}}+\frac{\mathcal{G}_n}{\sqrt{4n \gamma'_n}}, 
			\qquad \gamma'_n:=\frac{\log n-5\log\log n-\log (2\pi^4)}{2},\label{gamma}\\
			\max_{i=1}^n \{|\sigma_i|\} \stackrel{\mathrm{d}}{=}& 1+\sqrt{\frac{\gamma_n }{4n}}+\frac{\mathcal{G}_n}{\sqrt{4n \gamma_n}}, 
			\qquad \gamma_n:= \log n-2\log \log n-\log 2 \pi,\label{informal}	  
		\end{align}
		where $\mathcal{G}_n$ is an asymptotic Gumbel random variable as $n\to\infty$, \ie $\P (\mathcal{G}_n\le t) \to \exp(-\frac{\beta}{2} e^{-t})$. In particular, the limiting distribution has an exponential right-tail:
 \begin{align}\label{tail_complex}
			\lim_{n\rightarrow \infty}\P\Big( \max_{i=1}^n \Re\sigma_i \ge 1+\sqrt{\frac{\gamma'_n }{4n}}+\frac{t}{\sqrt{4n \gamma'_n}}\Big) =  \frac{\beta}{2} e^{-t} +o(1), \qquad t\rightarrow \infty.
		\end{align}
  %in the complex case and $\P (\mathcal{G}_n\le x) \to \exp(-\ee^{-x/2})$ in the real case. 
%  A similar asymptotic expansion also holds for the rightmost eigenvalue~\cite{Bender,AkemannP,maxRe_Gin} with a different parameter
%		\begin{equation}\label{gamma}
%			\gamma'_n=\frac{\log n-5\log\log n-\log (2\pi^4)}{2}.
%		\end{equation}
	We remark that \cite{maxRe_Gin} also proves an effective estimate on the right-tail asymptotics for any $1\ll t \ll \sqrt{\log n}$.
%and the same right-tail asymptotic hold for the spectral radius as well. 
It is worth mentioning that the Gumbel fluctuations in (\ref{gamma})--(\ref{informal}) have been proved very recently in \cite{gumbel} for general matrices with i.i.d.~entries under a mild moment condition. Following this, the optimal speed of convergence was obtained in \cite{Ma}. We also refer to the survey papers \cite{BF22a,BF22b} on the recent progress on the Ginibre ensembles and the references therein. 
	%	In contrast to these typical estimates, we will study the probability of rare events, \ie the deviations of the extremal eigenvalues from their typical positions.
		
	\smallskip	
		
	 It was well known that the joint probability density of the eigenvalues $\{\sigma_i\}_{i=1}^n$ of the complex Ginibre ensemble $X$ is given by~\cite{Ginibre}
\begin{equation}\label{complex_pdf}
    p_n(z_1,\ldots,z_n): = \frac{n^n}{\pi^n 1!\cdots n!} \exp\Bigl(-n\sum_{i}\abs{z_i}^2\Bigr) \prod_{i<j}\bigl(n\abs{z_i-z_j}^2\bigr).
\end{equation}
Indeed, the eigenvalues form a determinantal point process, whose density function is given by
		\begin{align}\label{det_pp}
			p_n(z_1,\ldots,z_n)
			%: =& \frac{n^n}{\pi^n 1!\cdots n!} \exp\Bigl(-n\sum_{i}\abs{z_i}^2\Bigr) \prod_{i<j}\bigl(n\abs{z_i-z_j}^2\bigr)\nonumber\\
			=&\frac{n^n}{\pi^n n!} e^{-n\sum_{i}\abs{z_i}^2} \det\Bigl( \mathcal{K}_n(z_i, z_j)\Bigr)_{i,j=1}^n,  \qquad \mathcal{K}_n(z,w):=\sum_{l=0}^{n-1}\frac{(n z\ov{w})^l}{l!}.
		\end{align}
%		with an explicit correlation kernel function
%		\begin{equation}
%			\mathcal{K}_n(z,w):=\sum_{l=0}^{n-1}\frac{(n z\ov{w})^l}{l!}=e^{nz \ov w} \frac{\Gamma(n,n z \ov w)}{\Gamma(n)},
%		\end{equation}
%	where $\Gamma(n,z)=\int_{z}^\infty t^{n-1} e^{-t} \dd t$ is the incomplete Gamma function.	
In particular, Kostlan observed \cite{Kostlan92} that the collection of moduli of the eigenvalues of $\sqrt{2n}X$
		%$\{|\sqrt{n}\sigma_{i}|\}_{i=1}^n$ 
		has the same distribution as the collection of independent chi-distributed random variables with even degrees, \ie
		% $\{\chi_{2k}\}_{k=1}^n$, 
		\begin{align}\label{kostlan}
			\{\sqrt{2n}|\sigma_1|,\sqrt{2n}|\sigma_2|,\cdots, \sqrt{2n}|\sigma_n|\} \stackrel{\mathrm{d}}{=} \{\chi_{2},\chi_{4}, \cdots, \chi_{2n}\}.
		\end{align}
%	In other words, the distribution of the spectral radius of $\sqrt{2n}X$ coincides with the maximum of independent chi-distributed variables with even degrees. 
	%By direct computations, 
	Heuristically, the large deviations of the spectral radius of $\sqrt{2n}X$ are mainly from that of the chi-distributed random variable with the largest degree, \ie
	\begin{align}
	 \P\Big( \max_{i=1}^n |\sigma_i| \ge t \Big) \approx \P\Big( \chi_{2n} \ge \sqrt{2n} t \Big) \approx e^{-n\big(t^2-2 \log t-1\big)}, \qquad t \geq 1.
    \end{align}
	A formal large deviation estimate for the density of the spectral radius of the complex Ginibre ensemble has already been obtained \cite{complex_LDP0,complex_LDP} in the Coulomb gas setting with general potentials, from which an LDP for the spectral radius follows immediately.

		\smallskip

		However, the real Ginibre ensemble is largely different, whose conditional joint densities of eigenvalues on the number of real eigenvalues were found in \cite{Edel, LS91}. In contrast to (\ref{complex_pdf})--(\ref{det_pp}), the eigenvalues instead form a (conditional) Pfaffian point process with an explicit correlation kernel given by \cite{BS09,tail_0}. In particular, Kostlan's observation in (\ref{kostlan}) unfortunately fails. For the real Ginibre ensemble, there typically exist $O(\sqrt{n})$ real eigenvalues~\cite{real_ev}, and the asymptotic statistics of the real eigenvalues are very different from those of the complex eigenvalues~\cite{BS09}, which behave similarly to the eigenvalues of the complex Ginibre ensemble.
		%	It is easy to check that all the eigenvalues come in conjugated pairs. 
		%While the large deviation estimates for the probability that there are $O(n)$ and $O(1)$ real eigenvalues were obtained in \cite{real_large} and \cite{real_small} respectively, 
		 The limiting distribution of the largest real eigenvalue was derived \cite{RS14} in term of a Fredholm determinant, with a closed form expression given by \cite{BB20}. In particular, it has Gaussian right-tail asymptotics~\cite{real_tail,tail_0}, \ie 
		\begin{align}\label{tail_real}
			\lim_{n\rightarrow \infty}\P\Big( \max_{\sigma_i \in \R} \sigma_i \ge 1+\frac{t}{\sqrt{n}}\Big)= \frac{e^{-t^2}}{4\sqrt{\pi} t}+o(1), \quad t\rightarrow \infty,
		\end{align}
		which is quite different from the exponential right-tail of the complex eigenvalues shown in (\ref{tail_complex}). After we submitted the first version of this paper to the journal, we found that the LDP for the rightmost eigenvalue was already derived in \cite{arXiv:1709.04021} even for the more general elliptic ensemble. Large deviations for the number of real eigenvalues of the real elliptic Ginibre ensemble were also studied in \cite{ABL25,BJLS25}.

	\smallskip

    The first main result of this paper is the following LDPs for the spectral radius and rightmost eigenvalue of both the real and the complex Ginibre ensembles.  The formula (\ref{rate_func}) below agrees with  \cite[Eq. (34)]{complex_LDP0} for the spectral radius in the complex case and \cite[Theorem 4.1]{arXiv:1709.04021} for the rightmost eigenvalue in the real case. The other results seem to be new.

		\begin{theorem}[LDPs for extremal eigenvalues]\label{thm:main}
			For the real~$(\beta=1)$ and complex~$(\beta=2)$ Ginibre ensemble, both the spectral radius $\max_{i=1}^n\{|\sigma_i|\}$ and the rightmost eigenvalue $\max_{i=1}^n \Re \sigma_i$ satisfy the LDP with speed $n$ and good rate function 
			\begin{align}\label{rate_func}
				I_{\beta}(t)=
				\begin{cases}
				\frac{\beta}{2} \big( t^2-2 \log t-1\big), \quad \qquad &t\ge 1,\\
				+\infty, \quad\qquad &t<1.
				\end{cases}
			\end{align} 
		Moreover, for the real Ginibre ensemble ($\beta=1$), the extremal complex-valued eigenvalue $\max_{\sigma_i\in \C\setminus \R}|\sigma_i|$ satisfies the LDP with speed $n$ and rate function $2I_1(t)$, while the largest real eigenvalue $\max_{\sigma_i\in \R}|\sigma_i|$ satisfies 
$$\lim_{n\to\8} \frac1n\log \P\Big(\max_{\sigma_i \in \R} \sigma_i \ge t\Big) = -I_1(t), \qquad t\geq 1.$$
%\cor[What can we say about the real eigenvalue for $t<1$?] \nc
		\end{theorem}

\begin{remark}
 As indicated above, in the large deviation regime, the largest real eigenvalue of the real Ginibre ensemble dominates the largest complex eigenvalue. This explains to some extent the `Saturn effect' observed from the numerical simulations in Figures \ref{fig:realandcomplex} and \ref{fig:complexvsreal}, which was first reported in \cite[Section 5]{RS14}. Hence large deviation estimates in (\ref{rate_func}) are dramatically different from the typical fluctuation estimates  in (\ref{gamma})--(\ref{informal}), where the largest real eigenvalue in (\ref{tail_real}) is typically smaller than the largest complex eigenvalue by comparing their centers and fluctuations.
\end{remark}

\begin{figure}[htbp]
	\includegraphics[width=14cm]{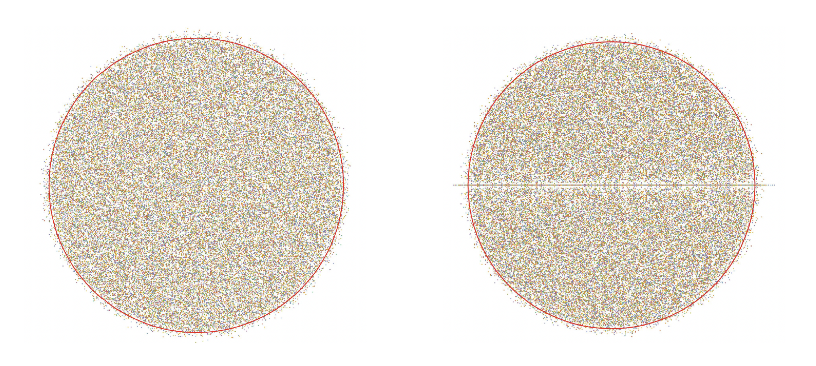}
	\centering
	\caption{We sample 300 complex Ginibre matrices (left) and real Ginibre matrices (right) of size $200\times 200$ and plot their eigenvalues on the same figure.  It seems that the rightmost eigenvalue of the real Ginibre ensemble is from real eigenvalues. This phenomenon, called `Saturn effect', might be caused by some rare events that happen during 300 samplings (represented by different colors), where the largest real eigenvalue is much larger than the complex eigenvalues.}
	\label{fig:realandcomplex}
\end{figure}

\begin{figure}[htbp]
	\includegraphics[width=10cm]{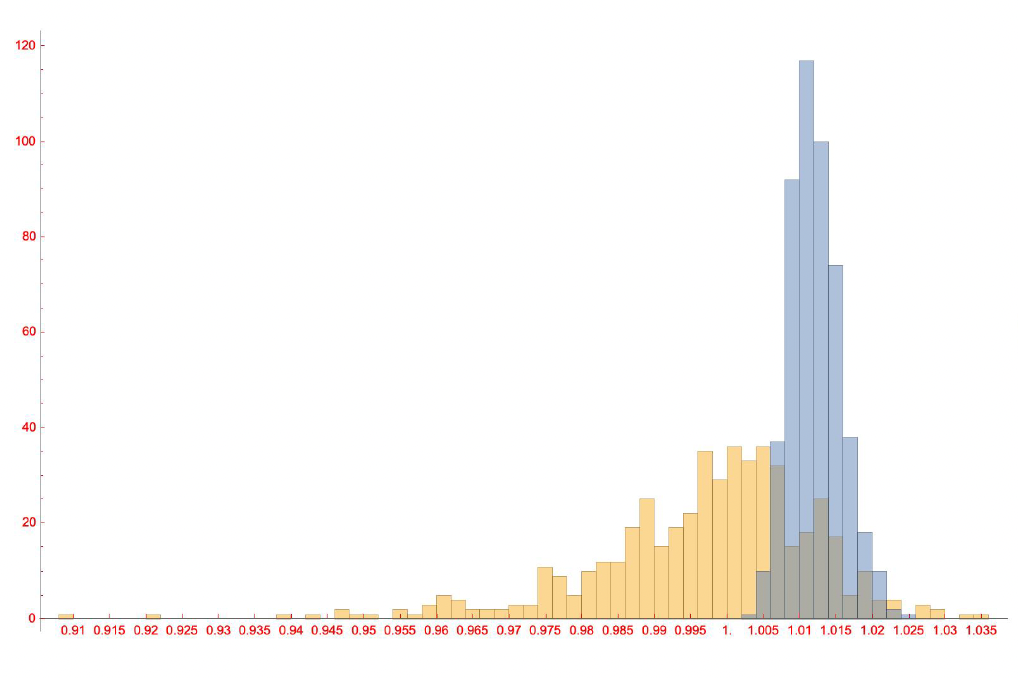}
	\centering
	\caption{ 
    We sample 500 real Ginibre matrices of size $5000 \times 5000$ and plot the histogram for the largest real eigenvalues (orange) and the largest complex eigenvalues in modulus (blue).  It seems that more real eigenvalues are deviating from the threshold 1.03 than the complex eigenvalues.}
	\label{fig:complexvsreal}
\end{figure}
 
\nc

\smallskip	

Moreover, we derive the following moderate deviation estimates for the extremal eigenvalues of Ginibre ensembles, which bridges the gap between the large deviation regime in Theorem~\ref{thm:main} and the typical fluctuating regime in (\ref{gamma})--(\ref{informal}) and (\ref{tail_real}).

\begin{theorem}[Moderate deviations]\label{MDP}
	Consider the real~$(\beta=1)$ and complex~$(\beta=2)$ Ginibre ensemble. For any $\sqrt{\frac{\gamma_n}{4n}} \leq d_n \ll 1$ with $\gamma_n$ in (\ref{informal}), 
there exist constants $C_1,C_2>0$ such that
	\begin{align}\label{mdp_result_1}
		&\P \Big( \max_{i=1}^n |\sigma_i| \geq 1+d_n \Big) \leq  \frac{C_1}{\sqrt{n} (d_n)^2} e^{-2n(d_n)^2 (1-O(d_n))}+  \frac{C_2}{\sqrt{n}d_n} e^{-n(d_n)^2(1-O(d_n))}\one_{\beta=1}.
	\end{align}
For any $\sqrt{\frac{\gamma'_n}{4n}} \leq d_n \ll 1$ with $\gamma'_n$ in (\ref{gamma}), 
there exist constants $C'_1,C'_2>0$ such that
	\begin{align}\label{mdp_result_2}
		&\P \Big( \max_{i=1}^n \Re \sigma_i \geq 1+d_n \Big) \leq   \frac{C_1'}{n (d_n)^{5/2}} e^{-2n(d_n)^2 (1-O(d_n))}+  \frac{C_2'}{\sqrt{n}d_n} e^{-n(d_n)^2(1-O(d_n))} \one_{\beta=1}.
	\end{align}
In particular, in the real case~($\beta=1$), for any $n^{-1/2} \ll d_n \ll 1$, there exists $C>0$ such that
\begin{align}\label{result_real}
	&\P \Big( \max_{\sigma_i \in \R} \sigma_i \geq 1+d_n \Big) \leq  \frac{C}{\sqrt{n}d_n} e^{-n(d_n)^2(1-O(d_n))},
\end{align}
%	and for any $B_n \ll \sqrt{\frac{\gamma_n}{4 n}} $ with $B_n \rightarrow 0$, 
%	\begin{align}
%		\P \Big( \rho_n \leq 1+B_n \Big) \lesssim ???.
%	\end{align}
and the same also holds true for $\max_{\sigma_i \in \R} |\sigma_i|$. Finally, for any $\sqrt{\log n/n} \ll d_n \ll 1$, we have
\begin{align}\label{MDP_formal}
	\lim_{n \rightarrow \infty} \frac{1}{n (d_n)^2}\log \P \Big( \max_{i=1}^n |\sigma_i| \geq 1+t d_n \Big)=-\beta t^2, \qquad t>0.
\end{align}
The same statement as in (\ref{MDP_formal}) also holds true for $\max_{i=1}^n \Re \sigma_i$.

\end{theorem}
We remark that the deviation estimates in (\ref{result_real}) for the real eigenvalues ($\beta=1$) are consistent with the Gaussian right-tail asymptotics of the limiting distribution as stated in (\ref{tail_real}).

\medskip

Finally, we obtain the following small deviations for the second and third terms in the asymptotic expansions (\ref{gamma})--(\ref{informal}). These deviation estimates are at most polynomially small and are indeed universal for general i.i.d.~matrices. The universality follows from the proof introduced in \cite{gumbel,maxRe} with some modifications. More precisely, we consider any matrix $X$ with i.i.d.~entries $x_{ab}\stackrel{\mathrm{d}}{=}n^{-1/2}\chi$, where $\chi$ satisfies $\E \chi=0$, $\E |\chi|^2=1$, additionally $\E \chi^2=0$ in the complex case, and $\chi$ has all the finite moments $$\E\big|\chi\big|^k\le C_k, \qquad k\in \N,$$ 
for some constants $C_k>0$.  To simplify the arguments, we also assume the probability measure $\chi$ has a density function $\phi$, and there exist some constants $a,b>0$ such that 
\begin{equation}\label{eq:hmb}
\phi \in L^{1+a}(\C), \qquad\qquad  \|\phi\|_{1+a} \leq n^{b}.
\end{equation}
Here we equip $\C$ with Lebesgue measure on $\R^2$.  
This condition ensures an effective lower bound on the smallest singular value that will be used in the proof following \cite{gumbel,maxRe}. This restriction can be removed using the argument in \cite[Section
6.1]{zbMATH06420677}; see more explanations in Remark 2.2 of \cite{maxRe}.
%Moreover, we also assume the following finite moment condition,

\begin{theorem}[Small deviations]\label{SDP}
	Given any real~$(\beta=1)$ or complex~$(\beta=2)$ i.i.d.~matrix defined as above, for any $s,t> 1$, there exists constants $C, C', C_s,C_t>0$ such that
	\begin{align}
		& \P \Big( \max_{i=1}^n |\sigma_i| \geq 1+s\sqrt{\frac{\gamma_n}{4 n}} \Big) \leq C (\log n)^{C_s} \Big(  n^{-\frac{s^2-1}{2}}+n^{-\frac{s^2}{4}}\one_{\beta=1} \Big), \label{small_complex1}\\
		& \P \Big( \max_{i=1}^n \Re \sigma_i \geq 1+t\sqrt{\frac{\gamma_n'}{4 n}} \Big)  \leq  C' (\log n)^{C_t}   \Big(  n^{-\frac{t^2-1}{4}}+n^{-\frac{t^2}{8}}\one_{\beta=1} \Big),\label{small_complex2}
	\end{align}
for sufficiently large $n$, with $\gamma_n$, $\gamma_n'$ given in (\ref{gamma})--(\ref{informal}). Moreover, for any $1\ll s_n \ll \log n$, 
\begin{align}\label{gumbel_tail}
\P \Big( \max_{i=1}^n |\sigma_i| \geq 1+\sqrt{\frac{\gamma_n}{4 n}} +\frac{s_n}{\sqrt{4n \gamma_n}} \Big) \leq C'' e^{-s_n},
\end{align}
and the same estimate also holds true for $\max_{i=1}^n \Re \sigma_i$ with $\gamma'_n$ in (\ref{gamma}).
%For the real Ginibre ensemble, the above holds true for the complex eigenvalues, while for the real eigenvalues, we have
%\begin{align}\label{result_real}
%	\P^{\mathrm{Gin}(\R)}\Big( \max_{\sigma_i \in \R} \sigma_i \geq 1+ r \sqrt{\frac{\log n}{n}} \Big) \lesssim  n^{-r^2}. 
%\end{align}
	%and for any $0<t<1$, we have
	%	\begin{align}
	%	\P \Big( \rho_n \leq 1+t\sqrt{\frac{\gamma_n}{4 n}} \Big) \lesssim ????.
	%\end{align}
\end{theorem}

 We remark that the tail bound in (\ref{gumbel_tail}) also agrees with the Gumbel right-tail asymptotics as stated in (\ref{tail_complex}). In the complex case~($\beta=2$), when evaluated beyond the threshold $1+\sqrt{\frac{\log n}{4n}}$, the right tail bound in (\ref{small_complex2}) is $n^{-1/4}$ smaller than that in (\ref{small_complex1}).
This $n^{-1/4}$ factor is mainly because the regime where the rightmost eigenvalue (maximum along a fixed direction) appears is almost $n^{-1/4}$ smaller than that of the spectral radius (maximum along all directions). However, it becomes different in the real case~($\beta=1$). If the evaluated threshold is beyond $1+\sqrt{\frac{\log n}{2n}}$, then the main contributions to the right-tail bounds are instead from the real eigenvalues, hence the right-tail bounds in (\ref{small_complex1})--(\ref{small_complex2}) are almost of the same order.
%\ie
%\begin{align}
%\P \Big( \max_{i=1}^n \Re \sigma_i \geq 1+s\sqrt{\frac{\log n}{4n}}\Big) \approx \P \Big( \max_{i=1}^n |\sigma_i| \geq 1+s\sqrt{\frac{\log n}{4n}}\Big).
%\end{align}

\bigskip
{\it Conventions and notations:} For $n$-dependent positive quantities $f_n,g_n$ we use the notation $f_n\ll g_n$ to denote that $\lim_{n\to\infty} (f_n/g_n)=0$.  For positive quantities $f,g$ we write $f\lesssim g$ and $f\sim g$ if $f \le C g$ or $c g\le f\le Cg$, 
respectively, for some constants $c,C>0$ independent from $n$. Throughout the paper $c,C>0$ (resp.~$c_t,C_t>0$) denote small and large absolute constants (resp.~constants depending only on $t$), respectively, which may change from line to line. 
%We denote vectors by bold-faced lower case Roman letters $\bm{x},\bm{y}\in\C^d$, for some $d\in\N$. Vector and matrix norms, $\lVert\bm{x}\rVert$ and $\lVert A\rVert$, indicate the usual Euclidean norm and the corresponding induced matrix norm. For any $d \times d$ matrix $A$ we use $\langle A\rangle:= d^{-1}\mathrm{Tr} A$ to denote the normalized trace of $A$. Moreover, for vectors ${\bm x}, {\bm y}\in\C^d$ we define the usual scalar product 
%$\langle \bm{x},\bm{y}\rangle:= \sum_{i=1}^d \overline{x_i} y_i. $
%Furthermore,  
%$\lVert f\rVert_\infty$ and  $\lVert f\rVert_1$  denote the $L^\infty$ and  $L^1$-norm of a function $f$, respectively. 
In the following we will use the notations $\P^{\mathrm{Gin}(\C)}$~(or $\E^{\mathrm{Gin}(\C)}$) and $\P^{\mathrm{Gin}(\R)}$~(or $\E^{\mathrm{Gin}(\R)}$) to denote the probability~(or expectation) for the complex and real Ginibre matrix, respectively.

\section{Proof of Theorem \ref{thm:main}}	

\subsection{Proof for the complex Ginibre Ensemble}

To prove Theorem \ref{thm:main} in the complex case, it suffices to prove the following proposition. 
	
		\begin{proposition}\label{prop_gin}
		For any $t<1$, we have
		\begin{align}\label{eq:ubdcp}
			\limsup_{n\to\8} \frac1n\log \P^{\mathrm{Gin}(\C)}(\max_{i=1}^n \{|\sigma_i|\}\le t)=-\8
		\end{align}
		and for any $t \geq 1$,
		\begin{align}\label{eq:main}
			\lim_{n\to\8} \frac1n\log \P^{\mathrm{Gin}(\C)}(\max_{i=1}^n \{|\sigma_i|\}\ge t) = -(t^2-2 \log t-1).
		\end{align}
	 The statements (\ref{eq:ubdcp})--(\ref{eq:main}) also hold true for $\max_{i=1}^n \Re \sigma_i$.
	\end{proposition}

Recall from~\cite{Mehta} that the eigenvalues of the complex Ginibre ensemble in (\ref{det_pp}) form a determinantal point process with $k$-point correlation function given by
\begin{align}\label{complex}
	p^{(k)}_n(z_1,\ldots,z_k)= 
	%\frac{(n-k)!}{n!}
	\det\Bigl( K_n(z_i, z_j)\Bigr)_{i,j=1}^k, \qquad K_n(z,w):= \frac{n}{\pi} e^{-n(\abs{z}^2+\abs{w}^2)/2} \mathfrak{e}_{n-1}(n z \ov w),
\end{align}
% From (\ref{det_pp}), the kernel function $K_n$ is given by
% \begin{equation}\label{kernel_complex}
% 	K_n(z,w)=\frac{n}{\pi} e^{-n(\abs{z}^2+\abs{w}^2)/2} \mathfrak{e}_{n-1}(n z \ov w),
% 	%= \frac{n}{\pi} e^{-n(\abs{z}^2+\abs{w}^2-2 z \ov w)/2} \frac{\Gamma(n,n z\ov{w})}{\Gamma(n)},
% \end{equation} 
 where the polynomial $\mathfrak{e}_{n-1}$ is defined by
\begin{align}\label{eq:poly}
	\mathfrak{e}_{n-1}(z):=\sum^{n-1}_{k=0} \frac{z^k}{k!}=e^{z} \frac{\Gamma(n,z)}{\Gamma(n)}, \qquad \Gamma(n,z):=\int_z^\infty t^{n-1} e^{-t}\dd t,
\end{align}
with the integration contour going from $z\in \C$ to the real infinity. These polynomials have the following precise asymptotic estimates.
\begin{lemma}[Proposition 3 in \cite{poly1}]\label{lemma:poly}
%	For a small enough $\delta>0$. let $\mu(z)=\sqrt{z-\log z-1}$ be analytic in $|z-1|<\delta$ with $\mu(1+x)>0$ for $0<x<\delta$, where $\log z$ is taken with a standard branch cut on $(-\infty,0]$. Then for any $M>1$, it holds that
%	\begin{align} \label{bbb}
%		e^{-nz} \mathfrak{e}_{n}(nz)= \frac{1}{2\sqrt{2} \mu'(z)} \mathrm{erfc}\big (\sqrt{n} \mu(z)\big) \left(1+O\Big(\frac{1}{n|z-1|}\Big) \right),
%	\end{align}
%	for any $z$ satisfying $\frac{M}{\sqrt{n}} \leq |z-1|\leq \delta $ and $\big|\arg(z-1)\big|\leq 2\pi/3$, where 
%	\begin{align}\label{ccc}
%		\mathrm{erfc}(z):=\frac{2}{\sqrt{\pi}} \int_{z}^\infty  e^{-t^2} \dd t =\frac{e^{-z^2}}{\sqrt{\pi} z} \left( 1-\frac{1}{2z^2}+O(|z|^{-4}) \right),
%	\end{align}
%	with $|\arg z| <3\pi/4$, as $|z| \rightarrow \infty$. Moreover, for $|z-1| \leq \frac{M}{\sqrt{n}}$, 
%	\begin{align} \label{fff}
%		e^{-nz} \mathfrak{e}_{n}(nz)= \frac{1}{2} \mathrm{erfc}\big (\sqrt{n} \mu(z)\big) \left(1+O\Big(\frac{1}{\sqrt{n}}\Big)\right).
%	\end{align}
%	In particular, 
Uniformly in $t> 0$, we have
	\begin{align}\label{aaa}
		e^{-nt} \mathfrak{e}_{n}(nt)=\one_{0\leq t<1} +\frac{1}{\sqrt{2}} \frac{\mu(t)t}{t-1} \mathrm{erfc}\big (\sqrt{n} \mu(t)\big) \left(1+O\Big(\frac{1}{\sqrt{n}}\Big)\right),
	\end{align}
with $\mu(t)=|t-1-\log t|^{1/2}$ and 
	\begin{align}\label{ccc}
		\mathrm{erfc}(t):=\frac{2}{\sqrt{\pi}} \int_{t}^\infty  e^{-s^2} \dd s =\frac{e^{-t^2}}{\sqrt{\pi} t} \left( 1-\frac{1}{2t^2}+O(|t|^{-4}) \right).
	\end{align}
\end{lemma}

	\begin{proof}[Proof of Proposition \ref{prop_gin}]

The first estimate (\ref{eq:ubdcp}) follows directly from the LDP for the empirical spectral measures in Theorem \ref{thm:ldp_esd}. Indeed, we may take a bounded Lipschitz function $f:\C\to\R$ supported in $\{z\in\C: \Re z \ge t\}$ such that $f(z)=1$ for $\Re z>\frac{t+1}{2}$ where $0<t<1$. By the LDP for the empirical spectral measures $\mu_n$,
\begin{align*}
    \P(\max_{i=1}^n \{|\sigma_i|\}\le t)\le \P\Big(\max_{i=1}^n \Re \si_i\le t\Big)\le \P\Big(\Big|\int f \dd \mu_n -\frac1{\pi}\int f(z)\dd^2 z\Big|>\delta\Big)\le e^{-c n^2}
\end{align*}
for some $\delta>0$, where $c>0$ depends on $f$ and $\delta$.

We next prove the second estimate (\ref{eq:main}) for $t \geq 1$. For $t=1$, it follows directly from (\ref{informal}) that $\P(\max_{i=1}^n \{|\sigma_i|\}\ge 1) \geq c$ for some constant $c>0$. We hence focus on proving (\ref{eq:main}) for $t>1$.  To obtain the upper bound, we write
		\begin{align}\label{upper}
			\P(\max_{i=1}^n \{|\sigma_i|\}\ge t)=&\P \big( \#\{1\leq i \leq n: |\sigma_i| \geq t\} \geq 1\big) \leq \E \big[\#\{1\leq i \leq n: |\sigma_i| \geq t\}\big].
		\end{align}
Note that the one point correlation function for the eigenvalue process is given by $K_n(z,z)$ in	(\ref{complex}). Hence for $t> 1$, we have
\begin{align}\label{K_bound}
		\E \big[\#\{1\leq i \leq n: |\sigma_i| \geq t\}\big]	&= \int_{|z| \geq t}  K_n(z,z) \dd^2 z=\int_{|z| \geq t} \frac{n}{\pi} e^{-n \abs{z}^2} \mathfrak{e}_{n-1}(n |z|^2) \dd^2 z\nonumber\\
		&=\int_{|z| \geq t} \frac{n}{\pi \sqrt{2 \pi n} } \frac{1}{|z|^2-1} e^{-n \big(|z|^2-2\log |z|-1\big)}\dd^2 z \Big(1+O(n^{-1/2})\Big)\nonumber\\
  &    =\sqrt{\frac{2}{\pi}}\int_t^\8 \frac{r}{r^2-1}e^{-n(r^2-2\log r-1)+\frac12\log n} \dd r  \Big(1+O(n^{-1/2})\Big)       \nc \nonumber \\
		&\leq  C_t e^{-n \big(t^2-2\log t-1\big)  - \nc\frac{1}{2} \log n},
\end{align}
where we also used that from Lemma \ref{lemma:poly} 
$$\mathfrak{e}_{n-1}(nz)  =\frac{(ez)^{n} }{\sqrt{2\pi n}(z-1)}\Big(1+O\Big(\frac1{\sqrt n}\Big)\Big),$$ 
 and the following estimate from integration by parts 
\begin{align}
    &\frac{t^2 e^{-n(t^2-2\log t)}}{2n(t^2-1)^2} \Big[1-\frac{t^2+1}{n(t^2-1)^2} \Big]\le \int_{t}^\8 \frac{r}{r^2-1} e^{-n(r^2-2\log r)} \dd r\le \frac{t^2 e^{-n(t^2-2\log t)}}{2n(t^2- 1)^2},\label{eq:est2}
\end{align}
for $t>1$. One could obtain a similar lower bound from \eqref{eq:est2} as in (\ref{K_bound}), \ie
\begin{align}\label{K_bound_lower}
	\E \big[\#\{1\leq i \leq n: |\sigma_i| \geq t\}\big]
	&\geq  C'_t e^{-n \big(t^2-2\log t-1\big)-\frac{1}{2} \log n}.
\end{align}
Combining (\ref{K_bound}) with (\ref{upper}), we obtain the desired upper bound 
\begin{align}\label{upper_bound}
\P(\max_{i=1}^n \{|\sigma_i|\}\ge t) \leq C_t e^{-n \big(t^2-2\log t-1\big)-\frac{1}{2} \log n}.
\end{align}
 To obtain the matching lower bound, we have
		\begin{align}\label{lower}
	\P(\max_{i=1}^n \{|\sigma_i|\}\ge t)=\P \big( \#\{1\leq i \leq n: |\sigma_i| \geq t\} \geq 1\big) &\geq \frac{1}{n} \E \big[\#\{1\leq i \leq n: |\sigma_i| \geq t\}\big]\nonumber\\
     &\geq C'_t e^{-n \big(t^2-2\log t-1\big) -\frac{3}{2} \log n} \nc,
\end{align}
where we also used (\ref{K_bound_lower}). This together with (\ref{upper_bound}) proves (\ref{eq:main}).

\medskip

To end the proof, we consider the rightmost eigenvalue. Similar to (\ref{upper})--(\ref{K_bound}), we have 
\begin{align}\label{rightmost_bound}
	\P\big( \max_{i=1}^n \Re \sigma_i \geq t  \big)
	\leq&\E \big[\#\{1\leq i \leq n: \Re \sigma_i \geq t\}\big] = \int_\R \int_{x \geq t}  K_n(x+\ii y,x+\ii y) \dd x \dd y\nonumber\\
	\leq & \int_{|z| \geq t}  K_n(z,z) \dd^2 z
%	 =&  \int_{\R} \int_{x\geq t} \frac{n}{\pi \sqrt{2 \pi n} } \frac{1}{x^2+y^2-1} e^{-n \big(x^2+y^2-\log(x^2+y^2)-1\big)} \dd x \dd y \Big(1+O(n^{-1})\Big)\nonumber\\
	\leq   C_t e^{-n \big(t^2-2\log t-1\big)-\frac{1}{2} \log n},
\end{align}	
for $t> 1$. One could obtain a similar lower bound as in (\ref{lower}), \ie 
		\begin{align}\label{lower_rt}
	\P\big(\max_{i=1}^n \Re \sigma_i\ge t\big)&
	%=\P \big( \#\{1\leq i \leq n: \Re \sigma_i \geq t\} \geq 1\big) 
	\geq  \frac{1}{n} \E \big[\#\{1\leq i \leq n: \Re \sigma_i \geq t\}\big]\nonumber\\
	&=\frac{1}{n} \int_{\R} \int_{x\geq t} \frac{n}{\pi \sqrt{2 \pi n} } \frac{1}{x^2+y^2-1} e^{-n \big(x^2+y^2-\log(x^2+y^2)-1\big)} \dd x \dd y \Big(1+O(n^{-1/2})\Big)\nonumber\\
	&\geq  C \frac{1}{\sqrt{n}}  \int_{x\geq t}  \left( \int_{\R} \frac{1}{x^2+y^2-1} e^{-ny^2}  \dd y \right) e^{-n \big(x^2-\log(x^2)-1\big)} \dd x  \nonumber\\
 &\ge C\frac1{n} \int_t^\8 \frac1{x^2-1} e^{-n \big(x^2-\log(x^2)-1\big)} \dd x \notag \nc\\
	&\geq  C'_t e^{-n \big(t^2-2\log t-1\big)- 2\log n},
\end{align}
where we used that $\log (x^2+y^2) \geq \log (x^2)$, and for sufficiently large $n$,
\begin{align*}\int_0^\8 \frac{1}{x^2+y^2-1}e^{-ny^2}\dd y=&\frac{2n}{\sqrt{x^2-1}} \int_{0}^{\infty} \arctan\Big(\frac{y}{\sqrt{x^2-1}}\Big)y e^{-ny^2} \dd y\ge \frac{C}{\sqrt{n}(x^2-1)},
\end{align*}
%the bound for $n>(t^2-1)^{-1}>(x^2-1)^{-1}$
%\[
%\int_0^\8 \frac{1}{x^2-1+y^2}e^{-ny^2}\dd y\ge \frac{1}{\sqrt{n}(x^2-1)} \int_0^\8 e^{-(1+n^{-1}(x^2-1)^{-1})y^2} \dd y \ge \frac{C}{\sqrt{n}(x^2-1)} %\cdot \sqrt{\frac{n(x^2-1)}{2(1+n(x^2-1))}}
%\]
together with the following estimates from integration by parts 
\begin{align}
  &\frac{te^{-n(t^2-2\log t)}}{2n(t^2-1)^2}\Big[1-\frac{3t^2+1}{2n(t^2-1)^2} \Big]\le \int_{t}^\8 \frac{1}{r^2-1} e^{-n(r^2-2\log r)} \dd r\le \frac{te^{-n(t^2-2\log t)}}{2n(t^2-1)^2},\label{eq:est3}  
\end{align}
for $t>1$. We hence finish the proof of Proposition~\ref{prop_gin}.
\end{proof}

\subsection{Proof for the real Ginibre Ensemble}
Given a real Ginibre matrix of dimension $n$ with $n=L+2M$, where $L$ is the number of real eigenvalues, and $M$ is the number of conjugated pairs of complex eigenvalues. The explicit formula for the joint density distribution of $L$ real eigenvalues and $M$ complex eigenvalues in the upper half plane was first introduced in \cite{LS91,Edel}. As the analog of (\ref{complex}), these eigenvalues indeed form a Pfaffian point process with the so-called $(l,m)$-correlation functions 
%given explicitly in 
\cite[Theorem 8]{BS09}~(or see \cite[Propositions 2.1--2.2]{RS14}). In particular, the one point correlation functions for the complex and real eigenvalue process are given by \cite{Edel,real_ev}, respectively,  \ie
%\begin{align}
%	p^{(l,m)}_n (x_1,\cdots, x_{l}, z_1, \cdots, z_m)=\mathrm{Pf} 
%	\begin{bmatrix}
%		\big( K_n^{\R,\R} (x_j, x_{j'})\big)_{j,j'=1}^{l} & \big( K_n^{\R,\C} (x_j, z_{k'})\big)_{j,k'} \\
%		\big( K_n^{\C,\R} (z_{k}, x_{j'})\big)_{j',k} & \big( K_n^{\C,\C} (z_k, z_{k'})\big)_{k,k'=1}^{m}
%	\end{bmatrix}.
%\end{align}
%More precisely, the complex/complex correlation kernel for the complex eigenvalues is given by
%\begin{align}\label{real}
%	\E \prod_{i=1}^n (1-g(\sigma_i)) & = \sum_{k=0}^{\lfloor n/2\rfloor} \frac{(-1)^k}{k!} \int_{\C^k} \mathrm{Pf}\Bigl( \sqrt{g(z_i)} K_n^{\C,\C}(z_i, z_j)\sqrt{g(z_j)}\Bigr)_{i,j=1}^k \dif^2 \bm z\nonumber\\
%	&= \bigl[\det(1-\sqrt{g}K_n^{\C,\C}\sqrt{g})\bigr]^{1/2},
%\end{align}
%for test functions \(g\colon\C\to[0,1]\) invariant under complex conjugation, \(g(\ov z)=g(z)\), and vanishing on the real line, \(g(x)=0\), \(x\in\R\), where 
\begin{align}
%	K_n^{\C,\C}(z,w):=\begin{pmatrix}
%		-\ii S_n(z,\ov w)& S_n(z, w)\\ 
%		-S_n(w, z)& \ii S_n(\ov z, w)
%	\end{pmatrix}, \qquad 
S^{\C,\C}_n(z,z):=&\frac{\sqrt{2}n^{3/2}|\Im z|}{\sqrt{\pi}}  e^{2n (\Im z)^2}  \erfc(\sqrt{2n}\abs{\Im z}) e^{-n|z|^2} \mathfrak{e}_{n-2}(n |z|^2),\label{realkernel_c}\\
S_n^{\R,\R}(x,x):=&\sqrt{\frac{n}{2\pi}} e^{-n x^2} \mathfrak{e}_{n-2}(nx^2)+
\frac{n^{\frac{n}{2}} |x|^{n-1} e^{-\frac{n}{2}x^2}}{2^{\frac{n}{2}} \Gamma\big(\frac{n}{2}\big) \Gamma\big(\frac{n-1}{2}\big) } \int_0^{\frac{nx^2}{2}}  u^{\frac{n-3}{2}} e^{-u} \dd u
%\frac{\sqrt{n} (\sqrt{n}x)^{n-1} e^{-\frac{n}{2}x^2}}{\sqrt{2\pi} (n-2)!} \int_0^{\sqrt{n}x} u^{n-2}e^{-\frac{1}{2}u^2} \dd u
.\label{realkernel_r}
\end{align}
We remark that these are also part of the Pfaffian kernels introduced in \cite{BS09}.
\nc

\begin{proposition}\label{prop_real}
%	For any $x<1$, we have
%	\begin{align}\label{eq:ubdcp_complex}
%		\limsup_{n\to\8} \frac1n\log \P^{\mathrm{Gin}(\R)}\Big(\max_{\sigma_i \in \C \setminus \R} \big|\sigma_i\big|\le x\Big)=-\8
%	\end{align}
%	and 
For any $t> 1$, we have
	\begin{align}
		\lim_{n\to\8} \frac1n\log \P^{\mathrm{Gin}(\R)}\Big(\max_{\sigma_i \in \C \setminus \R} \big|\sigma_i\big|\ge t\Big) =& -\big(t^2-2 \log t-1\big),\label{eq:main_complex}\\
	\lim_{n\to\8} \frac1n\log \P^{\mathrm{Gin}(\R)}\Big(\max_{\sigma_i \in \R} \sigma_i \ge t\Big) =& -\frac{t^2-2 \log t-1}{2}.\label{eq:main_real}
\end{align}
The same statement as in (\ref{eq:main_complex}) also holds true for $\max_{\sigma_i \in \C \setminus \R} \Re\sigma_i$, and (\ref{eq:main_real}) also holds true for $\max_{\sigma_i \in \R} |\sigma_i|$.
\end{proposition}

\begin{proof}[Proof of Proposition \ref{prop_real}]
The proof is similar to that of Proposition \ref{prop_gin}. We first consider the complex eigenvalues in (\ref{eq:main_complex}). From Lemma \ref{lemma:poly} note that
\begin{align}\label{kernel_real}
	S^{\C,\C}_n(z,z)&=\sqrt{2n\pi} e^{2n (\Im z)^2} |\Im z| \erfc(\sqrt{2n}\abs{\Im z}) K_{n-1}(z,z) e^{1-|z|^2}\Big( 1+O(n^{-1/2})\Big)\nonumber\\
	&=e^{1-|z|^2} K_{n-1}(z,z)\left(1 + O\left(\min\Big\{1,\frac{1}{n(\Im z)^2}\Big\}+\frac{1}{\sqrt{n}}\right)\right),
\end{align}\nc
with $K_{n}$ given by (\ref{complex}) in the complex case, where we also used the asymptotic estimate $\erfc(x)=e^{-x^2}/(\sqrt{\pi}x)(1+O(x^{-2}))$ and the bound $\erfc(x)\le e^{-x^2}/(\sqrt{\pi}x)$. This, together with the arguments as in (\ref{upper})--(\ref{lower_rt}), implies (\ref{eq:main_complex}) and its analog for $\max_{\sigma_i \in \C \setminus \R} \Re\sigma_i$. 

We next focus on the real eigenvalues in (\ref{eq:main_real}). Since the one point correlation function for the real eigenvalue process in (\ref{realkernel_r}) is symmetric, the statement for $\max_{\si_i\in\R}|\si_i|$ follows immediately from the one-sided estimate in (\ref{eq:main_real}). Note that the first term on the right side of (\ref{realkernel_r}) is given by 
\begin{align}\label{term1}
	\sqrt{\frac{n}{2\pi}} e^{-n x^2} \mathfrak{e}_{n-2}(nx^2)=\sqrt{\frac{\pi}{2n}} K_{n-1}(x,x) e^{1-x^2} \Big( 1+O(n^{-1/2})\Big).
\end{align}
%The corresponding integral can be bounded similarly as in (\ref{K_bound}) and (\ref{K_bound_lower}) with an additional small factor $n^{-1/2}$, 
Using that $\Gamma(z)\Gamma\big(z+1/2\big)=2^{1-2z}\sqrt{\pi} \Gamma(2z)$, we write the second term on the right side of (\ref{realkernel_r})
\begin{align}\label{star_term}
\frac{n^{\frac{n}{2}} x^{n-1} e^{-\frac{n}{2}x^2}}{2^{\frac{n}{2}} \Gamma\big(\frac{n}{2}\big) \Gamma\big(\frac{n-1}{2}\big) } \int_0^{\frac{nx^2}{2}}  u^{\frac{n-3}{2}} e^{-u} \dd u=\frac{2^{\frac{n}{2}-2} n^{\frac{n}{2}}  x^{n-1} e^{-\frac{n}{2}x^2}}{\sqrt{\pi} (n-2)!} \Big(\int_0^{\infty} -\int_{\frac{n}{2}x^2}^\infty \Big) u^{\frac{n-3}{2}}e^{-u} \dd u.
\end{align}
Note that the function $u^{\frac{n-3}{2}}e^{-\frac{n-3}{n}u}$ is decreasing in the regime $u\geq nx^2/2$ with $x\geq 1$. Thus the last integral over $[nx^2/2,\infty)$ in (\ref{star_term}) is bounded from above by
$$\int_{\frac{n}{2}x^2}^\infty  u^{\frac{n-3}{2}}e^{-u} \dd u \leq \Big( \frac{nx^2}{2}\Big)^{\frac{n-3}{2}} e^{-\frac{(n-3)x^2}{2}} \int_{\frac{n}{2}x^2}^\infty e^{-\frac{3}{n}u} \dd u \leq Cn\Big( \frac{nx^2}{2}\Big)^{\frac{n-3}{2}} e^{-\frac{(n-3)x^2}{2}},$$
which is much smaller than the whole integral over $[0,\infty)$. Thus by Stirling's formula, we have
\begin{align}\label{term2}
\frac{n^{\frac{n}{2}} x^{n-1} e^{-\frac{n}{2}x^2}}{2^{\frac{n}{2}} \Gamma\big(\frac{n}{2}\big) \Gamma\big(\frac{n-1}{2}\big) } \int_0^{\frac{nx^2}{2}}  u^{\frac{n-3}{2}} e^{-u} \dd u &= \frac{2^{\frac{n}{2}-2} n^{\frac{n}{2}}  x^{n-1} e^{-\frac{n}{2}x^2}}{\sqrt{\pi} (n-2)!}  \Gamma\big(\frac{n-1}{2} \big) (1+o(1))\nonumber\\
 &\sim  \frac1{x} e^{-\frac{n}{2} \big(x^2-2\log x-1\big)+\frac12 \log n}.
\end{align}
Combining (\ref{term1}), (\ref{term2}) with (\ref{realkernel_r}), we have
\begin{align}\label{S_bound}
	S_n^{\R,\R}(x,x)
	%=&\frac{\sqrt{n}}{\sqrt{2\pi}} e^{-n x^2} \mathfrak{e}_{n-2}(nx^2)+\frac{\sqrt{n} (\sqrt{n}x)^{n-1} e^{-\frac{n}{2}x^2}}{\sqrt{2\pi} (n-2)!} \int_0^{\sqrt{n}x} u^{n-2}e^{-\frac{1}{2}u^2} \dd u \nonumber\\
	\sim &\sqrt{\frac{\pi}{2n}} K_{n-1}(x,x)e^{1-x^2} + \frac{1}{x}e^{-\frac{n}{2} \big(x^2-2\log x-1\big)+\frac12 \log n} .
\end{align}
 Therefore, we obtain an upper bound similarly to (\ref{K_bound}), \ie
\begin{align}
	\E \big[\#\{\sigma_i\in \R: \sigma_i \geq t\}\big]&=\int_{x>t} S^{\R,\R}_n(x,x) \dd x \nonumber\\
 &\leq C \int_t^\8  \frac{1}{x^2-1}e^{-n(x^2-2\log x -1)} \dd x +
	C\sqrt{n}\int_{t}^\8 \frac1x e^{-\frac{n}{2} \big(x^2-2\log x-1\big)} \dd x
	 \nonumber\\
	%&\leq	C\sqrt{n}\int_{x>t}e^{-\frac{n}{2} \big(x^2-2\log x-1\big)} +O\Big( e^{-n \big(t^2-2\log t-1\big)-\frac{1}{2} \log n}\Big)\nonumber\\
	&\leq  C_t e^{-\frac{n}{2} \big(t^2-2\log t-1\big)-\frac12\log n},\label{eq:ubreal}
\end{align} 
where we also used (\ref{eq:est3}) together with the following estimate
\begin{align}\label{eq:est4}
    \frac{e^{-\frac{n}{2}(t^2-2\log t)}}{n(t^2-1)}\Big(1-\frac{2t^2}{n(t^2-1)^2} \Big)\le \int_t^\8 \frac1x e^{-\frac{n}{2} \big(x^2-2\log x\big)} \dd x \le \frac{e^{-\frac{n}{2}(t^2-2\log t)}}{n(t^2-1)}.
    \end{align}
From here, we also find a lower bound in the form of \eqref{eq:ubreal}. \nc 
Using that
	$$\frac{1}{n}\E \big[\#\{\sigma_i\in \R: \sigma_i \geq t\}\big] \leq \P\Big(\max_{\sigma_i \in \R} \sigma_i\ge t\Big) \leq \E \big[\#\{\sigma_i\in \R: \sigma_i \geq t\}\big],$$
we obtain the desired upper and lower bounds to prove (\ref{eq:main_real}) and  hence finish the proof of  Proposition \ref{prop_real}.
\end{proof}

\begin{proof}[Proof of Theorem \ref{thm:main} for $\beta=1$]
    As in the complex case, from the LDP for the empirical spectral measures in the real case we know
    \[
    \limsup_{n\to \8} \frac{1}{n} \log \P(\max_{i=1}^n \{|\sigma_i|\} \le t)\le \limsup_{n\to \8}\frac{1}{n} \log \P\Big(\max_{i=1}^n \Re \si_i \le t\Big) =-\8, \qquad t<1.
    \]
    The asserted LDP for $\rho_n$ follows from Proposition \ref{prop_real} since 
        \begin{align*}
    	\P\Big(\max_{\si_i\in \R} \si_i \ge t\Big) \le  \P(\max_{i=1}^n \{|\sigma_i|\} \ge t) \leq \P\Big(\max_{\si_i\in \R} \si_i \ge t\Big)+\P\Big(\max_{\si_i\in \C \setminus \R} |\si_i| \ge t\Big).
    \end{align*}
  For the rightmost eigenvalue, we observe that 
    \begin{align*}
         \P\Big(\max_{\si_i\in \R} \si_i \ge t\Big) \le \P\Big(\max_{i=1}^n \Re \si_i \ge t\Big) \le  \P(\max_{i=1}^n \{|\sigma_i|\}\ge t).
    \end{align*}
    The proof is completed using Proposition \ref{prop_real} again.
\end{proof}

\section{Proofs of Theorems \ref{MDP} and \ref{SDP}}

\subsection{Proof of Theorem \ref{MDP}}
We first consider the complex Ginibre ensemble. Similar to the arguments used in (\ref{upper})--(\ref{upper_bound}), for any $ \sqrt{\frac{\gamma_n}{4 n}} \leq  d_n \ll 1$, we have
\begin{align}\label{integral_K}
	\P^{\mathrm{Gin}(\C)}\Big( \max_{i=1}^n |\sigma_i| \geq 1+d_n  \Big) &\leq  \E^{\mathrm{Gin}(\C)} \big[\#\{1\leq i \leq n: |\sigma_i| \geq 1+d_n\}\big]\nonumber\\
	&=  \int_{|z| \geq 1+d_n }  K_n(z,z) \dd^2  z \nonumber\\
	&=\int_{|z|\geq 1+d_n} \frac{n}{\pi \sqrt{2 \pi n} } \frac{1}{|z|^2-1} e^{-n \big(|z|^2-2\log |z|-1\big)} \dd^2 z \Big(1+O(n^{-1/2})\Big)\nonumber\\
	&\lesssim % \sqrt{n} \int_{|z|\geq 1+d_n}  \frac{1}{|z|-1} e^{-2n(|z|-1)^2}  \dd^2 z   \Big(1+O(n^{-1})\Big),
  \frac{1}{\sqrt{n} (d_n)^2 }e^{-n[(d_n)^2 +2d_n-2\log(1+d_n)]}
\end{align}
where we used the upper estimate of \eqref{eq:est2} in the last step. 
%$|z|^2-2\log |z|-1=2(|z|-1)^2+O\big((|z|-1)^3\big)$ and that $|z|^2-2\log |z|-1 \geq (|z|-1)^2$ for $|z|\geq 1$.
%$|z|^2-2\log |z|-1=2(|z|-1)^2+O\big((|z|-1)^3\big)$. 
%Changing the variable $r = \sqrt{n}(|z|-1)$, 
Using a simple Taylor expansion $\log(1+x)= x-\frac{x^2}{2}+O(x^3)$, we obtain that
\begin{align}\label{spec}
	\P^{\mathrm{Gin}(\C)}\Big( \max_{i=1}^n |\sigma_i| \geq 1+d_n  \Big) \lesssim & \frac{1}{\sqrt{n} (d_n)^2} e^{-2n(d_n)^2(1-O(d_n))}.
 %\sqrt{n} \int_{r\geq \sqrt{n} d_n}  \frac{e^{-2r^2}}{r}   \dd r \lesssim  \frac{1}{\sqrt{n} (d_n)^2} e^{-2n(d_n)^2}.
\end{align}\nc
This proves (\ref{mdp_result_1}) for the complex case~($\beta=2$).

The rightmost eigenvalue (\ref{mdp_result_2}) can be estimated similarly. For any $ \sqrt{\frac{\gamma_n'}{4 n}} \leq  d_n \ll 1$, we have
\begin{align}\label{integral_K_rt}
	\P^{\mathrm{Gin}(\C)}\big( \max_{i=1}^n \Re \sigma_i &\geq 1+d_n  \big)
	\leq \int_\R \int_{x \geq 1+d_n}  K_n(x+\ii y, x+\ii y) \dd x \dd y\nonumber\\
	=&  \int_{\R} \int_{x\geq 1+d_n} \frac{n}{\pi \sqrt{2 \pi n} } \frac{1}{x^2+y^2-1} e^{-n \big(x^2+y^2-\log(x^2+y^2)-1\big)} \dd x \dd y \Big(1+O(n^{-1/2})\Big)\nonumber\\
%	\lesssim  & \sqrt{n} \int_\R \int_{x\geq 1+d_n}  \frac{1}{x^2+y^2-1} e^{-n \big(x^2+y^2-\log(x^2+y^2)-1\big)} \dd x \dd y \nonumber\\
	\lesssim &\sqrt{n} \int_{|y| \leq 10\sqrt{d_n}} \int^{10}_{x\geq 1+d_n}  \frac{1}{x^2-1} e^{-n \big(x^2+y^2-\log(x^2+y^2)-1\big)} \dd x \dd y\nonumber\\
	& \qquad +O\Big(\frac{1}{\sqrt{n} (d_n)^2} e^{-2n(10d_n)^2(1-O(d_n))}\Big),
\end{align}
where the integral with $y \geq 10\sqrt{d_n}$ or $x\geq 10$ is bounded as in (\ref{integral_K})--(\ref{spec}) using that $|z|^2=x^2+y^2\geq 1+100d_n$. %Note that $|z|^2-2\log |z|-1=2(|z|-1)^2+O\big((|z|-1)^3\big)$. 
By a simple Taylor expansion for $y\leq 10\sqrt{d_n}$ and $ 1+d_n \leq x\leq 10 $, we have
\begin{align}
	x^2+y^2-\log(x^2+y^2)-1&= x^2+y^2-2\log x-\log\Big(1+\frac{y^2}{x^2}\Big)-1\nonumber\\
	&\geq x^2-2\log x-1+Cy^2(x-1).
\end{align}
Therefore we obtain from (\ref{integral_K_rt}) that
%A direct computation yields that
\begin{align}\label{rtmost}
	\P^{\mathrm{Gin}(\C)}\big( \max_{i=1}^n \Re \sigma_i \geq 1+d_n  \big)
	&\lesssim  \sqrt{n}  \int_{x\geq 1+d_n}  \frac{1}{x^2-1} e^{-n \big(x^2-2\log x-1\big)} \dd x \left(\int_{|y| \leq 10\sqrt{d_n}}e^{-C ny^2 d_n}\dd y\right)\nonumber\\
	&\qquad +O\Big(\frac{1}{\sqrt{n} (d_n)^2} e^{-2n(10d_n)^2(1-O(d_n))}\Big)\nonumber\\
	&\lesssim  \frac{1}{n (d_n)^{5/2}} e^{-2n(d_n)^2(1-O(d_n))},
\end{align}
where we also used the upper estimate in \eqref{eq:est3}.
%Thus the main contribution to the integral in (\ref{integral_K_rt}) is from the regime $x-1 =O( d_n)$ and $|y|=O\big(\frac{1}{\sqrt{nd_n}}\big)$
This proves (\ref{mdp_result_2}) in the complex case. 
%Again one could obtain a lower bound as in (\ref{integral_K_lower}), \ie 

We next consider the real case~($\beta=1$). Recall the one-point correlation function for the complex eigenvalue process in (\ref{kernel_real}). The same estimates as in (\ref{spec}) and (\ref{rtmost}) also hold true for the complex eigenvalues of the real Ginibre ensemble. We hence focus on the real eigenvalues. 
%From (\ref{S_bound}) and that $|z|^2-2\log |z|-1=2(|z|-1)^2+O\big((|z|-1)^3\big)$, for $ n^{-1/2} \ll x-1 \ll 1$, we have
% \begin{align}
% 	S^{\R,\R}_n(x,x) \lesssim n^{-1/2} K_{n-1}(x,x)+\sqrt{n} e^{-n(x-1)^2}.
% \end{align}
%Using the Taylor expansion for $\log(1+d_n)$, 
Similarly to \eqref{eq:ubreal}, we obtain that
\begin{align}\label{integral_K_real}
\P^{\mathrm{Gin}(\R)}\big( \max_{\sigma_i \in \R}
 \sigma_i \geq 1+d_n  \big) &\leq \E \big[\#\{\sigma_i\in \R: \sigma_i \geq 1+d_n\}\big]\nonumber\\
 &\lesssim  \frac{1}{nd_n^2}e^{-n[(d_n)^2+2d_n-2\log(1+d_n)]}+ \frac{1}{\sqrt{n} d_n} e^{-\frac{n}{2}[(d_n)^2+2d_n-2 \log(1+d_n)]} \notag\\
 &\le  \frac{1}{nd_n^2}e^{-2n(d_n)^2(1-O(d_n))}+ \frac{1}{\sqrt{n}d_n} e^{-n(d_n)^2(1-O(d_n))},
 %\int_{x \geq 1+d_n} S^{\R,\R}_n(x,x) \dd x \lesssim  \frac{1}{n (d_n)^2} e^{-2n(d_n)^2}+e^{-n(d_n)^2}.
\end{align}
where we also used the estimates in \eqref{eq:est3} and \eqref{eq:est4}.
Since the second term on the right side of \eqref{integral_K_real} is dominant, this proves (\ref{result_real}) and hence (\ref{mdp_result_2}) in the real case. One could obtain a similar upper bound for $\max_{\sigma_i \in \R}
|\sigma_i|$ as in (\ref{result_real}) and hence prove (\ref{mdp_result_1}) in the real case. 
%Moreover, one could obtain a lower bound as in (\ref{integral_K_lower}), \ie 

Finally, we prove the last statement (\ref{MDP_formal}). Using similar arguments as in (\ref{lower}) and (\ref{lower_rt}), we obtain the following (non-optimal) lower bound~(\cf (\ref{integral_K}), (\ref{rtmost}) and (\ref{integral_K_real}), respectively)
\begin{align}
	&\P^{\mathrm{Gin}(\C)}\Big( \max_{i=1}^n |\sigma_i| \geq 1+d_n  \Big) \geq \frac{1}{n} \E^{\mathrm{Gin}(\C)} \big[\#\{1\leq i \leq n: |\sigma_i| \geq 1+d_n\}\big]\nonumber\\
%	\gtrsim& \frac{1}{n^{3/2} (d_n)^2}e^{-n[(d_n)^2 +2d_n-2\log(1+d_n)]}\nonumber\\
	&\qquad \qquad\qquad\qquad\qquad\qquad\gtrsim  \frac{1}{n^{3/2} (d_n)^2} e^{-2n(d_n)^2(1-O(d_n))},\label{integral_K_lower}\\
	&\P^{\mathrm{Gin}(\C)}\big( \max_{i=1}^n \Re \sigma_i \geq 1+d_n  \big)\gtrsim  \frac{1}{n^2 (d_n)^{5/2}} e^{-2n(d_n)^2(1-O(d_n))},\label{rtmost_lower}\\
	&\P^{\mathrm{Gin}(\R)}\big( \max_{\sigma_i \in \R}
	\sigma_i \geq 1+d_n  \big) \gtrsim   \frac{1}{n^{3/2}d_n} e^{-n(d_n)^2(1-O(d_n))}.\label{integral_K_real_lower}
\end{align}
Hence for any $\sqrt{\log n/n} \ll d_n \ll 1$, using that $n(d_n)^2 \gg \log n$, the last statement (\ref{MDP_formal}) follows directly from combining the upper estimates (\ref{spec}), (\ref{rtmost})--(\ref{integral_K_real}) with the lower estimates (\ref{integral_K_lower})--(\ref{integral_K_real_lower}). We hence finish the proof of Theorem \ref{MDP}.

\subsection{Proof of Theorem \ref{SDP}}
The estimates in (\ref{small_complex1})--(\ref{gumbel_tail}) for the Ginibre ensembles follow directly from (\ref{mdp_result_1})--(\ref{mdp_result_2}), respectively. 
%The proof for the real Ginibre ensemble is similar with additionally (\ref{result_real}). 
In particular, along the proof we indeed obtain the following Ginibre estimates, \ie for $s,t>1$,
\begin{align}
	&\E^{\mathrm{Gin}} \Big[\#\big\{ |\sigma_i| \geq 1+s\sqrt{\frac{\gamma_n}{4 n}}\big\}\Big] \lesssim  \frac1{s^2}\Big(\frac{\log n}{\sqrt{n}}\Big)^{s^2-1}+ \frac{\big(\log n\big)^{\frac12(s^2-1)}}{s n^{\frac{s^2}{4}}}  \one_{\beta=1},\label{gin_number}\\
	%O\left( \Big( \frac{\log n}{n^{1/2}}\Big)^{p(s)} \right),
	&\E^{\mathrm{Gin}} \Big[\#\big\{ \Re \sigma_i \geq 1+t\sqrt{\frac{\gamma'_n}{4 n}}\big\}\Big]\lesssim \frac{1}{t^{\frac{5}{2}}}\Big(\frac{(\log n)^{\frac{5}{4}}}{n^{\frac{1}{4}}}\Big)^{t^2-1}+\frac{(\log n)^{\frac{5}{8}t^2-\frac{1}{2}}}{t n^{\frac{t^2}{8}}}\one_{\beta=1},\label{gin_number_rt}
	%O\left( \Big( \frac{(\log n)^{5/4}}{n^{1/4}}\Big)^{p(s)} \right),
\end{align}
and for any $1\ll s_n \ll \log n$, 
\begin{align}\label{gin_number_asy}
&\E^{\mathrm{Gin}} \Big[\#\big\{ |\sigma_i| \geq 1+\sqrt{\frac{\gamma_n}{4 n}}+\frac{s_n}{ 4n \gamma_n}\big\} \Big] \lesssim e^{-s_n},
%\quad \E^{\mathrm{Gin}} \Big[\#\big\{ \Re \sigma_i \geq 1+\sqrt{\frac{\gamma'_n}{4 n}}+\frac{t_n}{ 4n \gamma'_n}\big\} \Big] \lesssim e^{-t_n},
\end{align}
together with the same estimate for $\Re \sigma_i$ with $\gamma_n'$ in (\ref{gamma}), where $\E^{\mathrm{Gin}}$ denotes the expectation for the Ginibre ensembles. In the following lemma, we extend the above Ginibre results to any i.i.d.~matrices, from which Theorem \ref{SDP} follows.

%Theorem \ref{SDP} then follows from Lemma \ref{lemma_uni} and .
\begin{lemma}\label{lemma_uni}
	Under the same conditions as in Theorem \ref{SDP}, then for any $s,t>1$, there exist $C_s, C_t>0$ such that
	\begin{align}
		&\E \Big[\#\big\{|\sigma_i| \geq 1+s\sqrt{\frac{\gamma_n}{4 n}}\big\}\Big]  \lesssim (\log n)^{C_s} \Big(  n^{-\frac{s^2-1}{2}}+n^{-\frac{s^2}{4}}\one_{\beta=1} \Big),\label{spec_uni}\\
		 & \E \Big[\#\big\{\Re \sigma_i \geq 1+t\sqrt{\frac{\gamma'_n}{4 n}}\big\}\Big]  \lesssim  (\log n)^{C_t}   \Big(  n^{-\frac{t^2-1}{4}}+n^{-\frac{t^2}{8}}\one_{\beta=1} \Big).\label{rtmost_uni}
	\end{align}
 Moreover, the Ginibre estimates as in (\ref{gin_number_asy}) also hold true for i.i.d.~matrices.
%	\begin{align}\label{spec_uni}
%	&\left| \E \Big[\#\big\{|\sigma_i| \geq 1+s\sqrt{\frac{\gamma_n}{4 n}}\big\}\Big]-	\E^{\mathrm{Gin}} \Big[\#\big\{ |\sigma_i| \geq 1+s\sqrt{\frac{\gamma_n}{4 n}}\big\}\Big]  \right|\nonumber\\
%	&\qquad\qquad\qquad\qquad\qquad\qquad\qquad =O\left(n^{-c\epsilon}  \Big( n^{-\frac{s^2-1}{2}}+n^{-\frac{s^2}{4}}\one_{\beta=1} \Big)\right),
%	\end{align}
%for some constant $c>0$. The same also applies to the rightmost eigenvalue with $\gamma'_n$ in (\ref{gamma}) and an error term given by $O\Big(n^{-c\epsilon}  \big( n^{-\frac{s^2-1}{4}}+n^{-\frac{s^2}{8}}\one_{\beta=1} \big)\Big)$.
%	\begin{align}\label{rt_uni}
%	\left| \E \Big[\#\big\{\Re \sigma_i \geq 1+s\sqrt{\frac{\gamma'_n}{4 n}}\big\}\Big]-	\E^{\mathrm{Gin}} \Big[\#\big\{ \Re \sigma_i  \geq 1+s\sqrt{\frac{\gamma'_n}{4 n}}\big\}\Big]  \right|=O\Big(n^{-c\epsilon}  n^{-\frac{p(s)}{4}}\Big).
%\end{align} 
\end{lemma}

%In particular, for any $t\geq 1$, there exists a constant $C>0$ such that
%\begin{align}\label{spec}
%	\P\Big( \max_{i=1}^n |\sigma_i| \geq 1+ t\sqrt{\frac{\gamma_n}{4 n}} \Big) \leq C  \Big(\frac{\log n}{\sqrt{n}}\Big)^{t^2-1}.
%\end{align}
%
%In particular, for any $t\geq 1$, there exists a constant $C>0$ such that
%\begin{align}\label{rtmost}
%	\P\Big( \max_{i=1}^n \Re \sigma_i \geq 1+ t\sqrt{\frac{\gamma_n'}{4 n}} \Big)  \leq C  \Big(\frac{(\log n)^{5/4}}{n^{1/4}}\Big)^{t^2-1}.
%\end{align}
%
%In particular we have
%\begin{align}\label{real_poly}
%	\P\Big( \max_{\sigma_i \in \R} \sigma_i \geq 1+ s \sqrt{\frac{\log n}{n}} \Big) \lesssim  n^{-s^2}.
%\end{align}

We follow the same strategy introduced in \cite{gumbel, maxRe} to prove Lemma \ref{lemma_uni}. The proof is indeed much easier than \cite{gumbel, maxRe} since we only need to estimate the expected number of eigenvalues in a given regime without considering its fluctuations.  To simplify the arguments, we will focus on the complex case, and the real case can be handled similarly with several modifications~(see more explanations in \cite[Section 2.4]{maxRe}). 
%To simplify the arguments, we will focus on proving (\ref{spec_uni}) for the spectral radius.

In the proof, we will often use the concept of ``with very high probability'' for an $n$-dependent event, meaning that for any fixed $D>0$ the probability of the event is bigger than $1-n^{-D}$ if $n\ge n_0(D)$.  
%Moreover, we will use the following standard definition of {\it stochastic domination};   its
%standard asymptotic properties can be found in Proposition 6.5 in \cite{EY17}. 
%\begin{definition}\label{definition of stochastic domination}
%	Let $\mathcal{X}\equiv \mathcal{X}^{(N)}$ and $\mathcal{Y}\equiv \mathcal{Y}^{(N)}$ be two sequences of nonnegative random variables. We say~$\mathcal{Y}$ stochastically dominates~$\mathcal{X}$ if, for all (small) $\epsilon>0$ and (large)~$D>0$,
%	\begin{align}\label{prec}
%		\P\big(\mathcal{X}^{(N)}>N^{\epsilon} \mathcal{Y}^{(N)}\big)\le N^{-D},
%	\end{align}
%	for sufficiently large $N\ge N_0(\epsilon,D)$, and we write $\mathcal{X} \prec \mathcal{Y}$ or $\mathcal{X}=O_\prec(\mathcal{Y})$. %We also say that an event $\Xi$ holds with high probability if $\P(\Xi) \geq 1-N^{-D}$ for any large $D>0$.
%\end{definition}

\begin{proof}[Proof of Lemma \ref{lemma_uni} in the complex case]
We will only prove (\ref{spec_uni}) for the spectral radius, and the proofs of (\ref{rtmost_uni}) for the rightmost eigenvalue and the estimate as in (\ref{gin_number_asy}) are essentially the same with minor modifications, which will be explained at the end of the proof.

\smallskip

Following \cite{gumbel, maxRe}, we will first bound the expected number of eigenvalues in the domain $\big\{|z|\geq 1+s\sqrt{\frac{\gamma_n}{4n}}\big\}$ from below and above by a regularized quantity as stated in (\ref{relation}) below. The derivation of (\ref{relation}) is quite standard, but we still include it for the completeness of the proof.  Fix any sufficiently small $\tau >0$, we truncate the domain $\big\{|z|\geq 1+s\sqrt{\frac{\gamma_n}{4n}}\big\}$ by
$$\Omega_{s}:=\Big\{z\in \C: 1+s\sqrt{\frac{\gamma_n}{4n}} \leq |z| \leq 1+n^{-1/2+\tau}\Big\}, \quad s\geq 1.$$ 
Note that from \cite[Theorem 2.1]{AEK19}, there is no eigenvalue beyond $|z| \geq 1+n^{-1/2+\tau}$ with very high probability. It then suffices to count the expected number of eigenvalues in the annulus $\Omega_s$.
%Then for any large $D>0$, we have
%\begin{align}\label{prob}
%	\E \Big[\#\big\{ |\sigma_i| \geq 1+s\sqrt{\frac{\gamma_n}{4 n}}\big\}\Big] = \E\Big[\# \big\{\sigma_i \in \Omega_{s}\big\} \Big]+O(n^{-D}).
%\end{align}
To achieve this, we choose two test functions $f^{\pm}_{s} \in C^{2}_c(\C)$ to bound $\one_{\Omega_{s}}$ from below and above  
\begin{align}\label{test}
	\one_{\Omega^-_{s}}(z) \leq f^{-}_{s}(z) \leq \one_{\Omega_{s}}(z) \leq f^{+}_{s}(z) \leq \one_{\Omega^+_{s}}(z),
\end{align}
with $\Omega^{\pm}_{s}:=\{z\in \C: 1+s\sqrt{\frac{\gamma_n}{4n}}\mp\frac{1}{\sqrt{n}}\le|z| \leq 1+n^{-\frac{1}{2}+\tau}\}$ such that 
$$ \max_{\alpha_1,\alpha_2 \in \N, \alpha_1+\alpha_2\le2} \left\{ n^{-\frac{\alpha_1+\alpha_2}{2}} \max_{z\in \C} \Big|\partial^{\alpha_1}_{z} \partial^{\alpha_2}_{\bar z} f_s^{\pm}(z)\Big| \right\}=O(1).$$
It is straightforward to check with the inequalities in (\ref{test}) that
\begin{align}\label{relation}
\E\Big[\# \big\{\sigma_i \in \Omega^-_{s}\big\} \Big] &\leq \E\Big[\sum_{i=1}^n f^{-}_{s}(\sigma_i)\Big] \leq \E\Big[\# \big\{\sigma_i \in \Omega_{s}\big\} \Big] \nonumber\\
& \le\E\Big[\sum_{i=1}^n f^{+}_{s}(\sigma_i)\Big] \leq \E\Big[\# \big\{\sigma_i \in \Omega^+_{s}\big\} \Big].
\end{align}
Note that the inequalities in (\ref{relation}) are also true when $X$ is a Ginibre matrix. With our choices of test functions in (\ref{test}), we know that $\E^{\mathrm{Gin}(\C)}\Big[\sum_{i=1}^n f^{\pm}_{s}(\sigma_i)\Big]$ have the same upper bound as $\E^{\mathrm{Gin}}\big[\# \big\{\sigma_i \in \Omega_{s}\big\} \big]$ in (\ref{gin_number})--(\ref{gin_number_rt}), \ie
\begin{align}\label{number_complex_f}
	\E^{\mathrm{Gin}(\C)}\Big[\sum_{i=1}^n f^{\pm}_{s}(\sigma_i)\Big] \lesssim (\log n)^{C_s} n^{-\frac{s^2-1}{2}}, \qquad s\geq 1.
\end{align}
To extend the above Ginibre estimates in (\ref{number_complex_f})  to general i.i.d.~matrices, it then suffices to establish the following comparison principle: for any large $D>0$,
\begin{align}\label{comparison}
\left|\E\Big[\sum_{i=1}^n f^{\pm}_{s}(\sigma_i)\Big]-\E^{\mathrm{Gin}(\C)}\Big[\sum_{i=1}^n f^{\pm}_{s}(\sigma_i)\Big]  \right|=O\Big(n^{-c\epsilon}  n^{-\frac{s^2-1}{2}}+n^{-D}\Big).
\end{align}
Once we have proved (\ref{comparison}), combining this with the Ginibre estimate in (\ref{number_complex_f}) and using the inequalities (\ref{test})--(\ref{relation}), we complete the proof of (\ref{spec_uni}) in the complex case, \ie  
$$\E\big[\# \big\{\sigma_i \in \Omega_{s}\big\} \big] \lesssim (\log n)^{C_s} n^{-\frac{s^2-1}{2}}, \qquad s\geq 1.$$

\smallskip

In the remaining part, we aim to prove the comparison principle (\ref{comparison}) using the strategy introduced in the proof of \cite[Theorem 4.1]{gumbel}. Since the arguments are quite similar, we will only sketch the proof and address the main differences. Although \cite[Theorem 4.1]{gumbel} is formulated for test functions $f$ with supporting domains slightly different from (\ref{test}), the key ingredients obtained there are also effective for our choice of regime in (\ref{test}). For instance,  recall from \cite[Proposition 3.9]{gumbel} that, for any large $D>0$ and any small $\tau,\epsilon>0$ with $\tau< 2\epsilon$, the smallest singular value of $X-z$, denoted by $\lambda_1^z$, has the following precise tail asymptotics:
\begin{align}\label{tail_bound}
	\P\big(\lambda_1^z\leq E_0 \big) \lesssim n^{3/2} E_0^2 e^{-n (|z|^2-1)^2/2}+n^{-D}, \qquad E_0= n^{-3/4-\epsilon},
\end{align}
uniformly in $z$ such that $n^{-1/2} \ll |z|-1 \leq n^{-1/2+\tau}$, which covers our choice of $z$ from (\ref{test}). Moreover, the joint density bound on the smallest singular values of $X-z_i$ for different $z_i$ is also stated in \cite[Proposition 3.8]{gumbel} similarly.

For the reader's convenience, we first recall the definitions and notations used in \cite{gumbel}. Fix a large $D>0$ and set $T:=n^{D}$. By Girko's formula~\cite{Girko1984,zbMATH06420677}, for any test function $f\in C^{2}_c(\C)$, with very high probability,
\begin{align}\label{girko0}
	\LL_f:=\sum_{i=1}^n f(\sigma_i)=&-\frac{1}{4 \pi}  \int_{\C} \Delta_z f(z) \int_0^{T}  \Im \Tr G^z(\ii \eta) \dd \eta \dd^2 z +O(n^{-D}),
\end{align}
 where  $G^{z}$ is the {\it resolvent} or {\it Green function} of the Hermitized matrix~$H^{z}$:
\begin{align}\label{def_G}
	G^{z}(w):=(H^{z}-w)^{-1}, \qquad H^{z}:=\begin{pmatrix}
		0  &  X-z  \\
		X^*-\overline{z}   & 0
	\end{pmatrix},\qquad w \in \C\setminus \R.
\end{align}
 Recall from \cite[Section 3.1]{gumbel} (and the references therein) that the resolvent $G^{z}$ satisfies the following {\it local law}. There exists a small constant $c>0$ such that, with very high probability,
\begin{align}
\sup_{||z|-1|\leq c} \sup_{n^{-1} \leq \eta \leq 1} \left|\eta \Tr \big( G^z(\ii \eta)-M^z(\ii \eta) \big)  \right| \leq n^{\xi}, 
\end{align}
for any small $\xi>0$ and sufficiently large $n$, with a deterministic matrix $M^{z}\in \C^{2n \times 2n}$ given by
\begin{align}\label{Mmatrix}
	M^{z}(w):=\begin{pmatrix}
		m^{z}(w)  &  -zu^z(w)  \\
		-\overline{z}u^z(w)   & m^{z}(w)
	\end{pmatrix} \quad \mbox{with} \quad  u^z(w):=\frac{{m}^z(w)}{w+ {m}^z(w)},
\end{align}
where ${m}^z(w)$ is the unique solution to the scalar cubic equation with side condition
\begin{align}
	\label{m_function}
	-\frac{1}{{m}^z(w)}=w+{m}^z(w)-\frac{|z|^2}{w+{m}^z(w)}, \qquad\quad \mathrm{Im}[m^z(w)]\mathrm{Im}w>0.
\end{align}
We remark that the local law estimates for multiple products of $G^{z_i}$ are also stated in \cite[Theorems 3.3 and 3.5]{gumbel} and will be used in our proof with the choice of $z$ from (\ref{test}).

Set $\eta_c:=n^{-L}$ for a large $L>0$ to be fixed later. Then we define the following regularized quantity to approximate $\LL_f$ in absolute mean, \ie 
\begin{align}\label{eq:325}
	\wh \LL_{f} :=-\frac{1}{4\pi}  \int \Delta_z f(z) \Big(\int_{\eta_c}^{T}  \Im \Tr \big( G^z(\ii \eta)-M^z(\ii \eta) \big) \dd \eta \Big)  q_z \dd^2 z,
\end{align} 
where $q_z$ is a regularized truncating function introduced in \cite[Eq.~(6.4)]{gumbel} to control the lower bound of the smallest singular value $\lambda_1^{z}$ of $X-z$, \ie
\begin{align}\label{eq2}
	q_z:=q\left( \int_{-E_0 }^{E_0}  \Im \Tr G^z(y+\ii \eta_0) \dd y \right), 
\end{align}
  with $E_0:= n^{-3/4-\epsilon}$ and $\eta_0:=n^{-3\zeta}E_0$, choosing small $\epsilon,\zeta>0$ satisfying $\zeta<\epsilon/100$, where $q: \R_+ \rightarrow [0,1]$ is a smooth and non-increasing cut-off function such that
\begin{equation}\label{F_function}
	q(x)=1, \quad \mbox{if} \quad 0 \leq x \leq 1/9; \qquad q(x)=0, \quad \mbox{if} \quad x \geq 2/9.
\end{equation}
Repeating the same arguments as in Step 1--6 from \cite[Section 6]{gumbel}, we obtain the analog of \cite[Proposition 6.1]{gumbel}, \ie for any test function $f=f_s^{+}$ or $f=f_s^{-}$ chosen in (\ref{test}),
\begin{align}\label{L_reduce}
  \E\big|\LL_f-\wh \LL_f\big| =O\Big(n^{-c\epsilon}  n^{-\frac{s^2-1}{2}}+n^{-D}\Big),
\end{align}
for some small constant $c>0$ and any large $D>0$.  As explained in Step 2 of \cite[Section 4]{gumbel}, we choose $\eta_L=n^{-L}$ with a sufficiently large $L>100$ depending on the parameters $a,b>0$ from (\ref{eq:hmb}), to control the event where $\lambda^{z}_1$ is super close to zero. In Step 3 of \cite[Section 4]{gumbel}, it was shown that the effect of subtracting $M^z$ in \eqref{eq:325} is negligible by direct computations using \eqref{Mmatrix}. The main input to derive (\ref{L_reduce}) is the precise tail estimate in (\ref{tail_bound}) for the smallest singular value slightly below its typical size. Compared to \cite[Proposition 6.1]{gumbel}, we gain an additional smallness $n^{-\frac{s^2-1}{2}}$~($s\geq 1$) from using (\ref{tail_bound}) in the restricted domain  $|z|>1+s\sqrt{\gamma_n/(4n)}+O(n^{-1/2})$ from (\ref{test}).

In the following, we aim to compare the expected quantity $\E[\wh \LL_f]$ with the corresponding Ginibre expectation $\E^{\mathrm{Gin}(\C)}[\wh \LL_f]$, via a continuous interpolating flow. More precisely, we consider the Ornstein--Uhlenbeck matrix flow
\begin{align}\label{flow}
	\dd H^z_t=-\frac{1}{2} (H^z_t+Z)\dd t+\frac{1}{\sqrt{n}} \dd \mathcal{B}_t, \quad Z:=\begin{pmatrix}
		0  &  zI  \\
		\overline{z}I   & 0
	\end{pmatrix}, \quad 
	\mathcal{B}_t:=\begin{pmatrix}
		0  &  B_t  \\
		B^*_t   & 0
	\end{pmatrix},
\end{align}
with initial condition $H^z_{t=0}= H^z$ in (\ref{def_G}) with the i.i.d.~matrix $X$, where $B_t$ is an $n \times n$ matrix with i.i.d.~standard complex valued Brownian motion entries. The matrix flow $H_t^z$ interpolates between the initial matrix $H^{z}$
and the same  matrix 
with $X$ replaced with an independent complex Ginibre ensemble. We then take the time derivative of the time-dependent $\E[\wh\LL_f]$ and apply the cumulant expansion formula to evaluate how the expectation evolves along the matrix flow. Repeating the same arguments as in \cite[Section 7]{gumbel} with a minor modification on the regime of $z$, we obtain an analog of \cite[Proposition 6.2]{gumbel}, \ie  for any $f=f^\pm_s$
\begin{align}\label{derivative}
\Big|\E\big[\wh \LL_f\big]-\E^{\mathrm{Gin}(\C)}\big[\wh \LL_f\big] \Big|  = O\Big(n^{-1/4+C\epsilon}  n^{-\frac{s^2-1}{2}}+n^{-D}\Big),
\end{align}
for some constant $C>0$ and any large $D>0$. The main technical tool to derive (\ref{derivative}) is the local law estimates for the products of resolvents obtained in \cite[Theorems 3.3 and 3.5]{gumbel} together with the tail asymptotics (\ref{tail_bound}) for the smallest singular value. The proof is indeed much easier than that in \cite[Section 7]{gumbel}, since we only need the expectation of $\wh\LL_f$ instead of its fluctuations. For instance, another key ingredient \cite[Proposition 3.8]{gumbel} characterizing decorrelations between the smallest singular values of $X-z_1$ and $X-z_2$, will no longer be needed for our usage, though it remains valid. Again, compared to \cite[Proposition 6.2]{gumbel}, we achieve an additional smallness $n^{-\frac{s^2-1}{2}}$~($s\geq 1$) using (\ref{tail_bound}) for $|z|>1+s\sqrt{\gamma_n/(4n)}+O(n^{-1/2})$.  Combining (\ref{derivative}) with (\ref{L_reduce}), for a sufficiently small $\epsilon>0$, we obtain that
\begin{align}\label{GFT}
	\Big|\E\big[ \LL_f\big]-\E^{\mathrm{Gin}(\C)}\big[ \LL_f\big] \Big|  = O\Big(n^{-c\epsilon}  n^{-\frac{s^2-1}{2}}+n^{-D}\Big), \qquad f=f^\pm_s.
\end{align}
This proves the comparison principle in (\ref{comparison}) and hence (\ref{spec_uni}) for the spectral radius.
%Combining (\ref{GFT}) with (\ref{prob}) and (\ref{relation}), we obtain that
%\begin{align}
%	\P \Big( \max_{i=1}^n |\sigma_i| \geq 1+s\sqrt{\frac{\gamma_n}{4 n}} \Big) \leq \E^{\mathrm{Gin}}\Big[\# \big\{\sigma_i \in \Omega^-_{s}\big\} \Big]+O\Big(n^{-c\epsilon}  n^{\frac{1-s^2}{2}}\Big).
%\end{align}
%Recalling from (\ref{integral_K})-(\ref{spec}), we have
%$$\E^{\mathrm{Gin}}\Big[\# \big\{\sigma_i \in \Omega^-_{s}\big\} \Big] \lesssim \Big( \frac{\log n}{n^{1/2}}\Big)^{s^2-1}.$$
%This, together with the Ginibre estimate in (\ref{number_complex_f}), yields that
%\begin{align}\label{number_complex_ff}
%	\E\Big[\sum_{i=1}^n f^{\pm}_{s}(\sigma_i)\Big]   \lesssim (\log n)^{C_s} n^{-\frac{s^2-1}{2}}.
%\end{align}

 \smallskip

To prove (\ref{rtmost_uni}) for the rightmost eigenvalue, one could prove a similar comparison principle as in (\ref{comparison}) with a modified error bound $O\Big(n^{-c\epsilon}  n^{-(s^2-1)/4}\Big)$, using that $|z|>1+s\sqrt{\gamma'_n/(4n)}+O(n^{-1/2})$ with $\gamma'_n$ given by (\ref{gamma}) (see modification details in \cite{gumbel}). 

 \smallskip

Lastly, one could also extend the Ginibre estimates as in (\ref{gin_number_asy}) to i.i.d.~matrices 
%, using similar arguments above and that $|z|>1+\sqrt{\gamma_n/(4n)}+s_n/\sqrt{4n \gamma_n}$ with $1\ll s_n \ll \log n$. 
 by choosing slightly different domains $\Omega^{\pm}_{s_n}=\{z\in \C: 1+\sqrt{\frac{\gamma_n}{4n}}+ \frac{s_n}{\sqrt{4n \gamma_n}}\mp\frac{1}{\sqrt{n}} \le|z| \leq 1+n^{-\frac{1}{2}+\tau}\}$ in (\ref{test}). Repeating the above arguments using (\ref{tail_bound}) for $|z|>1+\sqrt{\gamma_n/(4n)}+s_n/\sqrt{4n \gamma_n}$ with $1\ll s_n \ll \log n$, we then obtain a similar comparison principle as in (\ref{comparison}) with error $O\Big(n^{-c\epsilon}  e^{-s_n} \Big)$.
%\begin{align}
%	\Big|\E\big[ \LL_f\big]-\E^{\mathrm{Gin}(\C)}\big[ \LL_f\big] \Big|  = O\Big(n^{-c\epsilon}  e^{-s_n} \Big), \qquad f=f^\pm_s.
%\end{align}
%Together with (\ref{relation}), we hence have extended the Ginibre estimate in (\ref{gin_number_asy}) to i.i.d.~matrices.

 \smallskip

We hence complete the proof of Lemma \ref{lemma_uni} in the complex case. 
%The proof in the real case~($\beta=1$) is similar, see more details in \cite[Section 2.4]{maxRe}).
\end{proof}

{\bf Acknowledgments:}  Y.X.\ was partially supported by Grant No.\ 2024YFA1013503 from the National Key R\&D Program of China. Q.Z.\ was partially supported by NSFC Grant No.\ 12571152. Both authors thank the anonymous referees for careful readings and constructive suggestions.

\bibliographystyle{plain}
\bibliography{ginibre}

\begin{bibdiv}
\begin{biblist}

\bib{ABL25}{article}{
      author={{Akemann}, Gernot},
      author={{Byun}, Sung-Soo},
      author={{Lee}, Yong-Woo},
       title={{The probability of almost all eigenvalues being real for the
  elliptic real Ginibre ensemble}},
        date={2025-10},
     journal={Nonlinearity},
      volume={38},
      number={10},
       pages={105015},
      eprint={2503.18310},
}

\bib{AEK19}{article}{
      author={Alt, Johannes},
      author={Erd\H{o}s, L\'{a}szl\'{o}},
      author={Kr\"{u}ger, Torben},
       title={Spectral radius of random matrices with independent entries},
        date={2021},
        ISSN={2690-0998,2690-1005},
     journal={Probab. Math. Phys.},
      volume={2},
      number={2},
       pages={221\ndash 280},
         url={https://doi.org/10.2140/pmp.2021.2.221},
      review={\MR{4408013}},
}

\bib{AkemannP}{article}{
      author={Akemann, G.},
      author={Phillips, M.~J.},
       title={The interpolating {A}iry kernels for the {$\beta=1$} and
  {$\beta=4$} elliptic {G}inibre ensembles},
        date={2014},
        ISSN={0022-4715,1572-9613},
     journal={J. Stat. Phys.},
      volume={155},
      number={3},
       pages={421\ndash 465},
         url={https://doi.org/10.1007/s10955-014-0962-6},
      review={\MR{3192169}},
}

\bib{BDG01}{article}{
      author={Ben~Arous, G.},
      author={Dembo, A.},
      author={Guionnet, A.},
       title={Aging of spherical spin glasses},
        date={2001},
        ISSN={0178-8051,1432-2064},
     journal={Probab. Theory Related Fields},
      volume={120},
      number={1},
       pages={1\ndash 67},
         url={https://doi.org/10.1007/PL00008774},
      review={\MR{1856194}},
}

\bib{BG97}{article}{
      author={Ben~Arous, G.},
      author={Guionnet, A.},
       title={Large deviations for {W}igner's law and {V}oiculescu's
  non-commutative entropy},
        date={1997},
        ISSN={0178-8051},
     journal={Probab. Theory Related Fields},
      volume={108},
      number={4},
       pages={517\ndash 542},
  url={https://urldefense.proofpoint.com/v2/url?u=https-3A__doi.org_10.1007_s004400050119&d=DwIGaQ&c=KXXihdR8fRNGFkKiMQzstu-8MbOxd1NuZkcSBymGmgo&r=xgjcP5L9SZD7lyUnYxs7d005TSYbpcjMlIILZCrBUmY&m=ik53W52nlHKGpSmqYxuDPEAn4pA9Q4ERYBfJxF8OAyE&s=lMHez1gAxmw-m4NU51PmIEWU1kcp8TgylcHp-TgcNG4&e=},
      review={\MR{1465640}},
}

\bib{Bai97}{article}{
      author={Bai, Z.~D.},
       title={Circular law},
        date={1997},
        ISSN={0091-1798,2168-894X},
     journal={Ann. Probab.},
      volume={25},
      number={1},
       pages={494\ndash 529},
         url={https://doi.org/10.1214/aop/1024404298},
      review={\MR{1428519}},
}

\bib{AZ98}{article}{
      author={Ben~Arous, G\'{e}rard},
      author={Zeitouni, Ofer},
       title={Large deviations from the circular law},
        date={1998},
        ISSN={1292-8100,1262-3318},
     journal={ESAIM Probab. Statist.},
      volume={2},
       pages={123\ndash 134},
         url={https://doi.org/10.1051/ps:1998104},
      review={\MR{1660943}},
}

\bib{BB20}{article}{
      author={Baik, Jinho},
      author={Bothner, Thomas},
       title={The largest real eigenvalue in the real {Ginibre} ensemble and
  its relation to the {Zakharov}-{Shabat} system},
    language={English},
        date={2020},
        ISSN={1050-5164},
     journal={Ann. Appl. Probab.},
      volume={30},
      number={1},
       pages={460\ndash 501},
}

\bib{BCCT18}{article}{
      author={Bordenave, Charles},
      author={Caputo, Pietro},
      author={Chafa\"{\i}, Djalil},
      author={Tikhomirov, Konstantin},
       title={On the spectral radius of a random matrix: an upper bound without
  fourth moment},
        date={2018},
        ISSN={0091-1798,2168-894X},
     journal={Ann. Probab.},
      volume={46},
      number={4},
       pages={2268\ndash 2286},
         url={https://doi.org/10.1214/17-AOP1228},
      review={\MR{3813992}},
}

\bib{BCG22}{article}{
      author={Bordenave, Charles},
      author={Chafa\"{\i}, Djalil},
      author={Garc\'{\i}a-Zelada, David},
       title={Convergence of the spectral radius of a random matrix through its
  characteristic polynomial},
        date={2022},
        ISSN={0178-8051,1432-2064},
     journal={Probab. Theory Related Fields},
      volume={182},
      number={3-4},
       pages={1163\ndash 1181},
         url={https://doi.org/10.1007/s00440-021-01079-9},
      review={\MR{4408512}},
}

\bib{Bender}{article}{
      author={Bender, Martin},
       title={Edge scaling limits for a family of non-{H}ermitian random matrix
  ensembles},
        date={2010},
        ISSN={0178-8051,1432-2064},
     journal={Probab. Theory Related Fields},
      volume={147},
      number={1-2},
       pages={241\ndash 271},
         url={https://doi.org/10.1007/s00440-009-0207-9},
      review={\MR{2594353}},
}

\bib{BF22a}{article}{
      author={Byun, Sung-Soo},
      author={Forrester, Peter~J},
       title={Progress on the study of the ginibre ensembles i: Ginue},
        date={2022},
     journal={arXiv preprint arXiv:2211.16223},
}

\bib{BF22b}{article}{
      author={Byun, Sung-Soo},
      author={Forrester, Peter~J},
       title={Progress on the study of the ginibre ensembles ii: Ginoe and
  ginse},
        date={2023},
     journal={arXiv preprint arXiv:2301.05022},
}

\bib{poly1}{article}{
      author={Boyer, Robert},
      author={Goh, William M.~Y.},
       title={On the zero attractor of the {E}uler polynomials},
        date={2007},
        ISSN={0196-8858,1090-2074},
     journal={Adv. in Appl. Math.},
      volume={38},
      number={1},
       pages={97\ndash 132},
         url={https://doi.org/10.1016/j.aam.2005.05.008},
      review={\MR{2288197}},
}

\bib{BJLS25}{article}{
      author={{Byun}, Sung-Soo},
      author={{Jalowy}, Jonas},
      author={{Lee}, Yong-Woo},
      author={{Schehr}, Gr{\'e}gory},
       title={{Moderate-to-large deviation asymptotics for real eigenvalues of
  the elliptic Ginibre matrices}},
        date={2025-11},
     journal={arXiv e-prints},
       pages={arXiv:2511.09191},
      eprint={2511.09191},
}

\bib{BS09}{article}{
      author={Borodin, A.},
      author={Sinclair, C.~D.},
       title={The {G}inibre ensemble of real random matrices and its scaling
  limits},
        date={2009},
        ISSN={0010-3616,1432-0916},
     journal={Comm. Math. Phys.},
      volume={291},
      number={1},
       pages={177\ndash 224},
         url={https://doi.org/10.1007/s00220-009-0874-5},
      review={\MR{2530159}},
}

\bib{BY86}{article}{
      author={Bai, Z.~D.},
      author={Yin, Y.~Q.},
       title={Limiting behavior of the norm of products of random matrices and
  two problems of {G}eman-{H}wang},
        date={1986},
        ISSN={0178-8051,1432-2064},
     journal={Probab. Theory Related Fields},
      volume={73},
      number={4},
       pages={555\ndash 569},
         url={https://doi.org/10.1007/BF00324852},
      review={\MR{863545}},
}

\bib{maxRe_Gin}{article}{
      author={Cipolloni, Giorgio},
      author={Erd\H{o}s, L\'{a}szl\'{o}},
      author={Schr\"{o}der, Dominik},
      author={Xu, Yuanyuan},
       title={Directional extremal statistics for {G}inibre eigenvalues},
        date={2022},
        ISSN={0022-2488,1089-7658},
     journal={J. Math. Phys.},
      volume={63},
      number={10},
       pages={Paper No. 103303, 11},
         url={https://doi.org/10.1063/5.0104290},
      review={\MR{4496015}},
}

\bib{maxRe}{article}{
      author={Cipolloni, Giorgio},
      author={Erd\H{o}s, L\'{a}szl\'{o}},
      author={Schr\"{o}der, Dominik},
      author={Xu, Yuanyuan},
       title={On the rightmost eigenvalue of non-{H}ermitian random matrices},
        date={2023},
        ISSN={0091-1798,2168-894X},
     journal={Ann. Probab.},
      volume={51},
      number={6},
       pages={2192\ndash 2242},
         url={https://doi.org/10.1214/23-aop1643},
      review={\MR{4666294}},
}

\bib{gumbel}{article}{
      author={Cipolloni, Giorgio},
      author={Erd{\H{o}}s, L{\'a}szl{\'o}},
      author={Xu, Yuanyuan},
       title={Universality of extremal eigenvalues of large random matrices},
        date={2023},
     journal={arXiv preprint arXiv:2312.08325},
}

\bib{complex_LDP0}{article}{
      author={Cunden, Fabio~Deelan},
      author={Mezzadri, Francesco},
      author={Vivo, Pierpaolo},
       title={Large deviations of radial statistics in the two-dimensional
  one-component plasma},
        date={2016},
        ISSN={0022-4715,1572-9613},
     journal={J. Stat. Phys.},
      volume={164},
      number={5},
       pages={1062\ndash 1081},
         url={https://doi.org/10.1007/s10955-016-1577-x},
      review={\MR{3534484}},
}

\bib{Edel}{article}{
      author={Edelman, Alan},
       title={The probability that a random real {Gaussian} matrix has
  {{\(k\)}} real eigenvalues, related distributions, and the circular law},
    language={English},
        date={1997},
        ISSN={0047-259X},
     journal={J. Multivariate Anal.},
      volume={60},
      number={2},
       pages={203\ndash 232},
}

\bib{real_ev}{article}{
      author={Edelman, Alan},
      author={Kostlan, Eric},
      author={Shub, Michael},
       title={How many eigenvalues of a random matrix are real?},
        date={1994},
        ISSN={0894-0347,1088-6834},
     journal={J. Amer. Math. Soc.},
      volume={7},
      number={1},
       pages={247\ndash 267},
         url={https://doi.org/10.2307/2152729},
      review={\MR{1231689}},
}

\bib{tail_0}{article}{
      author={Forrester, Peter~J},
      author={Nagao, Taro},
       title={Eigenvalue statistics of the real ginibre ensemble},
        date={2007},
     journal={Physical review letters},
      volume={99},
      number={5},
       pages={050603},
}

\bib{arXiv:1709.04021}{misc}{
      author={Garcia, Xavier},
       title={On the number of equilibria with a given number of unstable
  directions},
         how={Preprint, {arXiv}:1709.04021 [math.{PR}] (2017)},
        date={2017},
         url={https://arxiv.org/abs/1709.04021},
}

\bib{Gem86}{article}{
      author={Geman, Stuart},
       title={The spectral radius of large random matrices},
        date={1986},
        ISSN={0091-1798,2168-894X},
     journal={Ann. Probab.},
      volume={14},
      number={4},
       pages={1318\ndash 1328},
  url={http://links.jstor.org/sici?sici=0091-1798(198610)14:4<1318:TSROLR>2.0.CO;2-O&origin=MSN},
      review={\MR{866352}},
}

\bib{Ginibre}{article}{
      author={Ginibre, Jean},
       title={Statistical ensembles of complex, quaternion, and real matrices},
        date={1965},
        ISSN={0022-2488,1089-7658},
     journal={J. Mathematical Phys.},
      volume={6},
       pages={440\ndash 449},
         url={https://doi.org/10.1063/1.1704292},
      review={\MR{173726}},
}

\bib{Girko1984}{article}{
      author={Girko, V.~L.},
       title={The circular law},
        date={1984},
        ISSN={0040-361X},
     journal={Teor. Veroyatnost. i Primenen.},
      volume={29},
      number={4},
       pages={669\ndash 679},
      review={\MR{773436}},
}

\bib{Gui23}{incollection}{
      author={Guionnet, Alice},
       title={Rare events in random matrix theory},
        date={2023},
   booktitle={I{CM}---{I}nternational {C}ongress of {M}athematicians. {V}ol. 2.
  {P}lenary lectures},
   publisher={EMS Press, Berlin},
       pages={1008\ndash 1052},
      review={\MR{4680275}},
}

\bib{Ma}{misc}{
      author={Hu, Xinchen},
      author={Ma, Yutao},
       title={Convergence rate of extreme eigenvalue of {Ginibre} ensembles to
  {Gumbel} distribution},
         how={Preprint, {arXiv}:2506.04560 [math.{PR}] (2025)},
        date={2025},
         url={https://arxiv.org/abs/2506.04560},
}

\bib{Kostlan92}{incollection}{
      author={Kostlan, Eric},
       title={On the spectra of {G}aussian matrices},
        date={1992},
      volume={162/164},
       pages={385\ndash 388},
         url={https://doi.org/10.1016/0024-3795(92)90386-O},
        note={Directions in matrix theory (Auburn, AL, 1990)},
      review={\MR{1148410}},
}

\bib{complex_LDP}{article}{
      author={Lacroix-A-Chez-Toine, Bertrand},
      author={Grabsch, Aur\'{e}lien},
      author={Majumdar, Satya~N.},
      author={Schehr, Gr\'{e}gory},
       title={Extremes of 2d {C}oulomb gas: universal intermediate deviation
  regime},
        date={2018},
        ISSN={1742-5468},
     journal={J. Stat. Mech. Theory Exp.},
      number={1},
       pages={013203, 39},
         url={https://doi.org/10.1088/1742-5468/aa9bb2},
      review={\MR{3761607}},
}

\bib{LS91}{article}{
      author={Lehmann, Nils},
      author={Sommers, Hans-J{\"u}rgen},
       title={Eigenvalue statistics of random real matrices},
        date={1991},
     journal={Physical review letters},
      volume={67},
      number={8},
       pages={941},
}

\bib{May72}{article}{
      author={May, Robert~M},
       title={Will a large complex system be stable?},
        date={1972},
     journal={Nature},
      volume={238},
      number={5364},
       pages={413\ndash 414},
}

\bib{Mehta}{book}{
      author={Mehta, Madan~Lal},
       title={Random matrices},
     edition={Third},
      series={Pure and Applied Mathematics (Amsterdam)},
   publisher={Elsevier/Academic Press, Amsterdam},
        date={2004},
      volume={142},
        ISBN={0-12-088409-7},
      review={\MR{2129906}},
}

\bib{HP1997}{incollection}{
      author={Petz, D\'{e}nes},
      author={Hiai, Fumio},
       title={Logarithmic energy as an entropy functional},
        date={1998},
   booktitle={Advances in differential equations and mathematical physics
  ({A}tlanta, {GA}, 1997)},
      series={Contemp. Math.},
      volume={217},
   publisher={Amer. Math. Soc., Providence, RI},
       pages={205\ndash 221},
         url={https://doi.org/10.1090/conm/217/02991},
      review={\MR{1606719}},
}

\bib{real_tail}{article}{
      author={Poplavskyi, Mihail},
      author={Tribe, Roger},
      author={Zaboronski, Oleg},
       title={On the distribution of the largest real eigenvalue for the real
  {G}inibre ensemble},
        date={2017},
        ISSN={1050-5164,2168-8737},
     journal={Ann. Appl. Probab.},
      volume={27},
      number={3},
       pages={1395\ndash 1413},
         url={https://doi.org/10.1214/16-AAP1233},
      review={\MR{3678474}},
}

\bib{R03}{incollection}{
      author={Rider, B.},
       title={A limit theorem at the edge of a non-{H}ermitian random matrix
  ensemble},
        date={2003},
      volume={36},
       pages={3401\ndash 3409},
         url={https://doi.org/10.1088/0305-4470/36/12/331},
        note={Random matrix theory},
      review={\MR{1986426}},
}

\bib{RS14}{article}{
      author={Rider, Brian},
      author={Sinclair, Christopher~D.},
       title={Extremal laws for the real {G}inibre ensemble},
        date={2014},
        ISSN={1050-5164,2168-8737},
     journal={Ann. Appl. Probab.},
      volume={24},
      number={4},
       pages={1621\ndash 1651},
         url={https://doi.org/10.1214/13-AAP958},
      review={\MR{3211006}},
}

\bib{TV10}{article}{
      author={Tao, Terence},
      author={Vu, Van},
       title={Random matrices: universality of {ESD}s and the circular law},
        date={2010},
        ISSN={0091-1798,2168-894X},
     journal={Ann. Probab.},
      volume={38},
      number={5},
       pages={2023\ndash 2065},
         url={https://doi.org/10.1214/10-AOP534},
        note={With an appendix by Manjunath Krishnapur},
      review={\MR{2722794}},
}

\bib{zbMATH06420677}{article}{
      author={Tao, Terence},
      author={Vu, Van},
       title={Random matrices: universality of local spectral statistics of
  non-{Hermitian} matrices},
    language={English},
        date={2015},
        ISSN={0091-1798},
     journal={Ann. Probab.},
      volume={43},
      number={2},
       pages={782\ndash 874},
}

\end{biblist}
\end{bibdiv}


\begin{thebibliography}{00}
			
			
			\bibitem{AkemannP} Akemann, G., Phillips, M.J.: 
			The Interpolating Airy Kernels for the $\beta=1$ and $\beta=4$ Elliptic Ginibre Ensembles. J. Stat. Phys. {\bf 155}, 421--465 (2014)
			
			
			
			\bibitem{AEK19} Alt, J., Erd{\H o}s, L., Kr\"uger, T.: 
			Spectral radius of random matrices with independent entries. Probability and Mathematical Physics. {\bf 2}(2), 221--280 (2021)
			
			\bibitem{Bai1997} Bai, Z.D.:  
			Circular law. Ann. Probab. {\bf 25}(1), 494--529 (1997)
			
			
			
			\bibitem{AZ98} Ben Arous, G. and Zeitouni, O.: Large deviations from the circular law. ESAIM: Probability and Statistics {\bf 2}, pp.123-134 (1998)
			
			
			
			
			
			\bibitem{Bender} Bender, M.: 
			Edge scaling limits for a family of non-Hermitian random matrix ensembles. Probab. Theory Related Fields. {\bf 147}, 241–271 (2010)
			
			
			%\bibitem{BCCT18} Bordenave, C., Caputo, P., Chafa\"i, D., Tikhomirov, K.: 
			%On the spectral radius of a random matrix: An upper bound without fourth moment. Ann. Probab. {\bf 46}(4), 2268--2286 (2018) 
			
			
			
%			\bibitem{BC12} Bordenave, C., Chafa\"i, D.: 
%			Around the circular law. Probab. Surv. {\bf 9}, 1--89 (2012)
			
			
			
			
			% \bibitem{BCG22} Bordenave, C., Chafa\"i, D. Garc\'ia-Zelada, D.: 
			% Convergence of the spectral radius of a random matrix through its characteristic polynomial. Probab. Theory Related
			%  Fields. {\bf 182}, 1163--1181 (2022)
			
			
			
			\bibitem{BS09} Borodin, A., Sinclair, C.D.: 
			The Ginibre Ensemble of Real Random Matrices and its Scaling Limits. Commun. Math. Phys. {\bf 291}, 177--224 (2009)
			
			
		
		\bibitem{poly2}  Bleher, P. and Mallison Jr, R., 2006. Zeros of sections of exponential sums. International Mathematics Research Notices, 2006, p.38937.
			
	
	        \bibitem{poly1} Boyer, R. and Goh, W.M., 2007. On the zero attractor of the Euler polynomials. Advances in Applied Mathematics, 38(1), pp.97-132.		
			
			
			\bibitem{BF22a} Byun, S. S., Forrester, P. J.: Progress on the study of the Ginibre ensembles I: GinUE. Preprint (2022). arXiv:2211.16223 
			
			\bibitem{BF22b} Byun, S. S., Forrester, P. J.: Progress on the study of the Ginibre ensembles II: GinOE and GinSE. Preprint (2023). arXiv:2301.05022 
			
			
%			\bibitem{Chblog} Chafa\"i, D.:
%			Around the circular law: an update. (Webblog) \\
%			{\tt https://djalil.chafai.net/blog/2018/11/04/around-the-circular-law-an-update.} Version 2018-11-04	
			
			
			
%			\bibitem{CP14} 	Chafa\"i, D., P\'ech\'e, S.: 
%			A note on the second order universality at the edge of Coulomb gases on the plane. J. Stat. Phys. {\bf 156}, 368--383 (2014)
			
			
			
			
			
			
			
%			\bibitem{CES21} Cipolloni, G., Erd{\H o}s, L., Schr\"oder, D.: Edge universality for non-Hermitian random matrices. Probability Theory and Related Fields, {\bf 179}(1), 1--28 (2021)
			
			
			
			
			\bibitem{maxRe_Gin} Cipolloni, G., Erd{\H o}s, L., Schr\"oder, D., Xu, Y.: 
			Directional extremal statistics for Ginibre eigenvalues. J. Math. Phys. {\bf 63}(11), 103303, (2022)
			
			
			\bibitem{maxRe} Cipolloni, G., Erd{\H o}s, L., Schr\"oder, D., Xu, Y.: 
			On the rightmost eigenvalue of non-Hermitian random matrices. Ann. Probab. {\bf 51}(6), 2192--2242 (2023)
			
			
			\bibitem{SpecRadius} Cipolloni, G., Erd{\H o}s, L., Xu, Y.:
			Precise asymptotics for the spectral radius of a large random matrix. Preprint (2022). arXiv:2210.15643 
			
			
		\bibitem{gumbel} 	Cipolloni, G., Erd{\H o}s, L. and Xu, Y.: Universality of extremal eigenvalues of large random matrices. arXiv preprint arXiv:2312.08325 (2023).
		
		
			\bibitem{complex_LDP0} F. D. Cunden, F. Mezzadri, P. Vivo, Large deviations of radial statistics in the two-dimensional one-component plasma, J. Stat. Phys. 164 (2016), 1062.
			
			
%		\bibitem{real_large} del Molino, L.C.G., Pakdaman, K., Touboul, J. et al. The Real Ginibre Ensemble with Real Eigenvalues. J Stat Phys 163, 303–323 (2016). 	
			
			
			%\bibitem{FG23} Fran\c{c}ois, Q., Garc\'ia-Zelada, D.: 
			%Asymptotic analysis of the characteristic polynomial for the Elliptic Ginibre Ensemble. Preprint (2023). arXiv:2306.16720 
			
			
			\bibitem{GinOE} A. Edelman, “The Probability that a Random Real Gaussian Matrix has k Real Eigenvalues, Related Distributions, and the Circular Law,” Journal of Multivariate Analysis, vol. 60, pp. 203– 232, 1997.
			
			\bibitem{real_ev} Edelman A, Kostlan E, Shub M. How many eigenvalues of a random matrix are real?. Journal of the American Mathematical Society. 1994;7(1):247-67.
			
%			\bibitem{EX22} Erd{\H o}s, L, Xu, Y.: 
%			Small deviation estimates for the largest eigenvalue of Wigner matrices. Bernoulli {\bf 29}(2), 1063--1079 (2023)
			
			 \bibitem{EY17}  Erd\H{o}s, L. and Yau, H.-T. 
			A Dynamical Approach to Random Matrix Theory. Courant Lecture Notes in Mathematics {\bf 28}. Providence: American Mathematical Society, 2017.
			
			
			\bibitem{tail_0} Forrester, P.J. and Nagao, T., 2007. Eigenvalue statistics of the real Ginibre ensemble. Physical review letters, 99(5), p.050603.
			
			
			\bibitem{Ginibre} Ginibre, J.:
			Statistical Ensembles of Complex, Quaternion, and Real Matrices. J. Math. Phys. {\bf 6}, 440--449 (1965) 
			
			
			
			\bibitem{Girko1984} Girko, V. L.: 
			Circular law. Teor. Veroyatnost. i Primenen. {\bf 29}, 669--679 (1984)
			
			
			
			\bibitem{book_fredholm} Gohberg, I., Goldberg, S. and Krupnik, N.:
			Traces and Determinants of Linear Operators. Operator Theory: Advances and Applications {\bf 116}, Birkh\"auser Basel, IX, 258 (2000)
			
		

		
		    \bibitem{HP1997} D. Petz and F. Hiai.: Logarithmic energy as an entropy functional, Advances in differential equations and mathematical physics (Atlanta, GA, 1997), Contemp. Math., vol. 217, Amer. Math. Soc., Providence, RI, 1998, p. 205–221.
		    
		    
		    
%		    \bibitem{real_small} Kanzieper, E., Poplavskyi, M., Timm, C., Tribe, R. and Zaboronski, O., 2016. What is the probability that a large random matrix has no real eigenvalues?.
		    
		    
		    \bibitem{poly3} Kriecherbauer, T., Kuijlaars, A.B.J., McLaughlin, K.D. and Miller, P.D., 2007. Locating the zeros of partial sums of exp(z) with Riemann-Hilbert methods. arXiv preprint arXiv:0709.1213.
		    
		    
		    
		    \bibitem{complex_LDP} Lacroix-A-Chez-Toine B, Grabsch A, Majumdar SN, Schehr G. Extremes of 2d Coulomb gas: universal intermediate deviation regime. Journal of Statistical Mechanics: Theory and Experiment. 2018 Jan 4;2018(1):013203.
		    
		    
		    \bibitem{LS91}		Lehmann, N. and Sommers, H.-J. (1991). Eigenvalue statistics of random real matrices. Phys. Rev. Lett. 67 941–944.
		    
		    
		   \bibitem{real_tail}  Poplavskyi, M., Tribe, R. and Zaboronski, O., 2017. On the distribution of the largest real eigenvalue for the real Ginibre ensemble.	
			
			\bibitem{Kostlan92} Kostlan, E.:
			On the spectra of Gaussian matrices. Linear Algebra Appl. {\bf 162/164}, Directions in matrix theory (Auburn, AL, 1990), 385--388 (1992)
			
			
			
			
			
			
			
			\bibitem{May72}  May, R. M.: 
			Will a large complex system be stable?  Nature {\bf 238}, 413--414 (1972)
			
			
			\bibitem{Metha} Mehta, M.: 
			Random Matrices. Pure and Applied Mathematics {\bf 142}, Third Edition, Academic Press (2004)
			
			
			
			
			
			
			\bibitem{R03} Rider, B.: 
			A limit theorem at the edge of a non-Hermitian random matrix ensemble. J. Phys. A {\bf 36}(12), 3401--3409 (2003)
			
			
			
			
			
			\bibitem{RS14} Rider, B., Sinclair, C. D.: 
			Extremal laws for the real Ginibre ensemble. Ann. Appl. Probab. {\bf 24}(4), 1621--1651 (2014)
			
			
			
			
			\bibitem{TV10} Tao, T., Vu, V.:
			Random matrices: universality of ESDs and the circular law. Ann. Probab. {\bf 38}(5), 2023--2065  (2010)
			
			
			
			
			
			
			
			
			
			
			
			
		\end{thebibliography}

	\end{document}